\documentclass[10pt, letterpaper, twoside]{article} 

\usepackage{amsmath, amsthm, amssymb,bbold,mathrsfs} 
\usepackage{bbold,mathrsfs} 
\usepackage{graphicx}
\usepackage[titletoc,title]{appendix}
\usepackage{enumitem}
\setlist[itemize]{itemsep=0pt,parsep=2pt,topsep=2pt}
\setlist[enumerate]{itemsep=0pt,parsep=2pt,topsep=2pt}
\usepackage[numbers]{natbib}

\usepackage[hidelinks]{hyperref}

\newtheorem{theorem}{Theorem}[section]

\newtheorem{prop}{Proposition}[section]
\newtheorem{lemma}{Lemma}[section]
\newtheorem{assumption}{Assumption}[section]

\theoremstyle{remark}
\newtheorem{rem}{Remark}[section]

\theoremstyle{definition}

\newtheorem{definition}{Definition}[section]

\theoremstyle{plain}

\numberwithin{equation}{section}

 \def\A{\mathbb{A}}
 \def\X{\mathbb{X}}
 \def\E{\mathbb{E}}
 
 \def\R{\mathbb{R}}
 
 \def\C{\mathcal{C}}
 
 \def\P{\mathcal{P}}
 \def\B{\mathcal{B}}

 \def\0{\mathbf{0}}
 \def\1{\mathbf{1}}
 \def\ind{\mathbb{1}}

\def\Pr{\mathbf{P}}
\def\proj{\mathop{\text{\rm proj}}}
\def\Z{\mathbb{Z}}
\def\Q0{{\mathbb{Q}_0}}

\def\W{\mathcal{W}}
\def\wto{\mathop{\overset{\text{\tiny \it w}}{\rightarrow}}}

\def\L{L}
\def\Me{\mathbb{M}}
\def\Fn{\mathbb{F}}
\def\m0{{\it 0}}
\def\urho{\underline{\rho}}
\def\nI{\mathcal{I}}
\def\Pstar{\text{$\text{P}^*$}}
\def\S{\mathcal{S}}
\def\J{\mathcal{J}}

\def\myqed{\hfill \qed}

\newcommand{\overbar}[1]{\mkern 2.5mu\overline{\mkern-3.5mu#1\mkern-0.7mu}\mkern 0.7mu}

\begin{document} 
\markboth{Linear Programming for Average-Cost MDPs}{} 

\title{On Linear Programming for Constrained and Unconstrained Average-Cost Markov Decision Processes with Countable Action Spaces and Strictly Unbounded Costs}
\author{Huizhen Yu\thanks{RLAI Lab, Department of Computing Science, University of Alberta, Canada (\texttt{janey.hzyu@gmail.com})}}
\date{}
\maketitle

\begin{abstract}
We consider the linear programming approach for constrained and unconstrained Markov decision processes (MDPs) under the long-run average cost criterion, where the class of MDPs in our study have Borel state spaces and discrete countable action spaces. Under a strict unboundedness condition on the one-stage costs and a recently introduced majorization condition on the state transition stochastic kernel, we study infinite-dimensional linear programs for the average-cost MDPs and prove the absence of a duality gap and other optimality results. Our results do not require a lower-semicontinuous MDP model. Thus, they can be applied to countable action space MDPs where the dynamics and one-stage costs are discontinuous in the state variable. Our proofs make use of the continuity property of Borel measurable functions asserted by Lusin's theorem.
\end{abstract}

\bigskip
\bigskip
\bigskip
\noindent{\bf Keywords:}\\
Markov decision processes; Borel state space; countable action space; average cost; constraints\\
minimum pair; majorization condition; infinite-dimensional linear programs; duality 

\clearpage
\tableofcontents
\clearpage

\section{Introduction}

We consider discrete-time Markov decision processes (MDPs) with the long-run average cost criterion. Our focus will be on the linear programming (LP) approach, for a class of unconstrained and constrained MDPs that have Borel state spaces, discrete countable action spaces, and unbounded one-stage costs.

LP methods for average-cost MDPs have a long history and an extensive literature. For MDPs with finite state and action spaces, see e.g., \cite{DeF68,HoK79,HoK84,Kar83};  for countable state spaces and countable or compact action spaces, see \cite{Bor88,Bor94,HoL94,HuK97,Las94}; and for Borel state and action spaces, see \cite{HG98,HGL03,HL94,HL02,KNH00,Yam75}. The interested reader may also consult the books \cite{Alt99,FS02,HL96,HL99,Puterman94} and their references.
The third group of results deal with uncountably infinite state spaces and are most closely related to our work. In particular, using the theory of infinite-dimensional LP (Anderson and Nash \cite{AnN87}), Hern\'{a}ndez-Lerma and Lasserre \cite{HL94} (see also \cite[Chap.\ 12]{HL99}, \cite{HL02}) formulated a general LP framework for Borel space average-cost MDPs. They studied the relations between the values of the primal/dual linear programs and the minimum average cost of an MDP, proved the absence of a duality gap under certain continuity conditions on the MDP model, and related the solutions of the programs to stationary optimal policies and average cost optimality equations (ACOE) of the MDP.
Much earlier than \cite{HL94}, Yamada \cite{Yam75} considered linear programs for a special class of geometrically ergodic MDPs with compact Euclidean state/action spaces and proved duality results for these problems. Building on the work \cite{HL94}, Hern\'{a}ndez-Lerma and Gonz\'{a}lez-Hern\'{a}ndez \cite{HG98} provided additional results and generalizations. Extensions of the LP method to constrained average-cost problems were studied by Kurano et al.~\cite{KNH00} for compact spaces and by Hern\'{a}ndez-Lerma et al.~\cite{HGL03} for non-compact spaces.

Another line of research that is closely related to our work, as well as to the prior work on LP mentioned above, is the minimum pair approach for average-cost Borel space MDPs (\cite{HLe93,Kur89}, \cite[Chap.\ 5.7]{HL96}; see also the related convex analytic approach \cite{Bor88}). 
With this approach, one considers minimizing the average cost over all policies and initial distributions, and the interest is in the existence of an optimal pair of policy and initial distribution with the following structure. The policy is stationary, and the associated initial distribution is an invariant probability measure of the Markov chain induced by the policy. In this paper we shall call a pair with such a structure a ``stationary pair'' and if it attains the minimum average cost, a ``stationary minimum pair.'' The feasibility and solvability of the primal linear programs studied in the prior work mentioned earlier in fact depend on the existence of such pairs. Conversely, a stationary minimum pair, when it exists, can be found by solving a linear program in the space of invariant probability measures induced by stationary policies, thus providing a way to find a stationary optimal policy for a subset of states.%footnote starts
\footnote{For finite state and action MDPs, Denardo \cite{Den70} seems to be the first to recognize the relation between the solution of a certain linear program and a stationary minimum pair, and he proposed to find a stationary average-cost optimal policy in a multichain MDP by repeatedly solving those linear programs on subproblems with smaller state spaces. This procedure is not applicable in general when the state space is uncountably infinite, since the ``chain structure'' of an MDP in this case can be complicated and hard to analyze. For some results on LP for ``multichain'' Borel space MDPs, see~\cite{HG98}.
}
%footnote ends
In some cases, with further ergodicity and regularity conditions, one can also extend the policy to an optimal one over the entire state space and establish stronger optimality, including sample-path optimality, of the policy \cite{HLe93,Kur89,Las99,VAm99}.
 
Our work builds upon earlier research on the LP and minimum pair methods for average-cost MDPs mentioned above. In those prior results the action space is more general than the countable action space we deal with in this paper. However, except for \cite{Yam75}, all of those results assume a lower-semicontinuous MDP model. Namely, they require the one-stage cost functions to be lower semicontinuous and the state transition stochastic kernels to be (weakly) continuous (\cite{Yam75} involves different continuity conditions; see Remark~\ref{rmk-Yamada} for details). Our work does not require this assumption.

We recently introduced in \cite{Yu19-minp} a majorization condition on the state transition stochastic kernel to deal with Borel space MDPs that do not satify such continuity conditions. For the case of countable action spaces (with the discrete topology), we obtained the existence of a stationary minimum pair and other average-cost optimality results analogous to those for lower-semicontinuous MDPs given by \cite{HLe93,Kur89,Las99,VAm99}. The purpose of the majorization condition is to make use of Lusin's theorem on the continuity of Borel measurable functions \cite[Thm.~7.5.2]{Dud02}. Roughly speaking, we require the existence of finite Borel measures on the state space that can majorize certain sub-stochastic kernels created from the state transition stochastic kernel, at all admissible state-action pairs (see Assumption \ref{cond-pc-3}(M)). We then use those majorizing finite measures in combination with Lusin's theorem to extract arbitrarily large (according to a given finite measure) sets on which certain Borel measurable functions involved in our analysis have desired continuity properties. With this technique, we are able to avoid the lower-semicontinuous model assumption and obtain results in \cite{Yu19-minp} that can be applied to MDPs with discontinuous dynamics and one-stage costs, although the application range is currently limited to the case of countable action spaces.

The purpose of this work is to further analyze the implications of the majorization condition and Lusin's theorem in the LP context, for both unconstrained and constrained MDPs.
The main contributions of this paper are as follows.
\begin{enumerate}[leftmargin=0.7cm,labelwidth=!]
\item[(i)]  For unconstrained average-cost MDPs, under the strictly unbounded cost condition and the majorization condition (cf.\ Assumption~\ref{cond-pc-3}), we prove there is no duality gap between the primal and dual linear programs in an LP formulation (see Theorem~\ref{thm-1}). 
\item[(ii)] For constrained average-cost MDPs, under conditions similar to those in (i), we first prove the existence of a stationary optimal pair and a stationary lexicographically optimal pair (which are analogous to stationary minimum pairs for unconstrained MDPs), and we then prove the absence of a duality gap for an LP formulation (see Theorems~\ref{thm-4.1} and~\ref{thm-4.2}, respectively). 
\end{enumerate}
In addition, we also discuss the maximizing sequences of dual linear programs and their relation with certain versions of ACOE (see Prop.~\ref{prp-2} for unconstrained MDPs and Props.~\ref{prp-4.3}, \ref{prp-4.4} for constrained MDPs).
Our results for unconstrained (resp., constrained) MDPs given in this paper can be compared with some of the prior results in \cite[Chap.\ 12]{HL99} and \cite{HL02} (resp., \cite{HGL03} and \cite{KNH00}) for lower-semicontinuous models. 

While this paper focuses on the average cost criterion, the analysis we give, with minor changes, can also be applied to constrained (or multi-objective) discounted-cost MDPs similar to those studied in~\cite{FeS96,HG00,HR04}, for finding constrained optimal or Pareto optimal policies (for a given initial distribution) using the LP approach, in the case of countable action spaces. 
In a separate recent work \cite{Yu20} based on similar ideas, we introduced another majorization condition for MDPs where both the state and action spaces are Borel, and used the majorization condition instead of the commonly required continuity/compactness conditions to prove the average cost optimality inequalities via the vanishing discount factor approach. 

The rest of this paper is organized as follows. In Section~\ref{sec-2} we give background materials about the average-cost MDP model, some prior optimality results for the minimum pair approach, and an overview of linear programs in topological vector spaces. In Section~\ref{sec-3} we present our LP formulation and duality results for unconstrained MDPs. We then extend these results to constrained MDPs in Section~\ref{sec-4}. Proofs for the theorems in Sections~\ref{sec-3} and~\ref{sec-4} are given in Section~\ref{sec-5}.

\section{Preliminaries} \label{sec-2}
We start with some notations and basic definitions. For a topological space $X$, $\B(X)$ denotes the Borel $\sigma$-algebra on $X$, and $\P(X)$ denotes the set of probability measures on $\B(X)$. We will refer to nonnegative or signed measures on $\B(X)$ as Borel measures. A \emph{Borel space} (a.k.a.\ \emph{standard Borel space}) is a separable metrizable space that is homeomorphic to a Borel subset of some Polish space (i.e., a separable and completely metrizable space) \cite[Chap.\ 7]{bs}. Let $X$ and $Y$ be Borel spaces. A \emph{Borel measurable stochastic kernel on $Y$ given $X$} is a Borel measurable function from $X$ into $\P(Y)$, where the space $\P(Y)$ is endowed with the topology of weak convergence. We denote the stochastic kernel by $q(dy \,|\, x)$. When it is continuous on $X$, we call it a \emph{continuous stochastic kernel} (it is also called \emph{weakly continuous} or \emph{weak Feller} in the literature).
For the space $\P(X)$ or more generally, the space of finite Borel measures on $X$, besides the topology of weak convergence just mentioned, we shall also consider other topologies in the next section when these spaces appear in infinite-dimensional linear programs.

We now introduce average-cost MDPs and the minimum pair approach, after which we will briefly review infinite-dimensional linear programs in topological vector spaces.

\subsection{MDP Model, Average Cost Criterion, and Minimum Pair Approach} \label{sec-2.1}

We consider an MDP with state space $\X$ and action space $\A$, where $\X$ is a Borel space and $\A$ is a countable space endowed with the discrete topology.   
The control constraint is specified by a set-valued map $A: \X \to 2^{\A}$. In particular, each state $x \in \X$ is associated with a nonempty set $A(x) \subset \A$ of admissible actions, and the graph of the map $A(\cdot)$, 
$$\Gamma := \{(x, a) \mid x \in \X, a \in A(x)\},$$ 
is assumed to be a \emph{Borel subset} of $\X \times \A$. 
If an action $a \in A(x)$ is taken at state $x$, a one-stage cost $c(x,a)$ is incurred, followed by a probabilistic state transition. 
We assume that the state transition is governed by a \emph{Borel measurable} stochastic kernel $q(dy \,|\, x, a)$ on $\X$ given $\X \times \A$, and that the one-stage cost function $c: \X \times \A \to [0, +\infty]$ is \emph{nonnegative and Borel measurable}, real-valued on $\Gamma$, and taking the value $+\infty$ outside $\Gamma$.

A policy is a sequence of stochastic kernels on $\A$ that specify how to take actions at each stage, given the history up to that stage. 
More precisely, for infinite-horizon average cost problems that we consider, a \emph{Borel measurable policy} is an infinite sequence $\pi : =(\mu_0, \mu_1, \ldots)$ where for each $n \geq 0$,
$\mu_n\big(da_n \!\mid x_0, a_0, \ldots, a_{n-1}, x_n \big)$ is a Borel measurable stochastic kernel on $\A$ given $(\X \times \A)^{n} \times \X$ and obeys the control constraint of the MDP:
$$ \mu_n\big(A(x_n) \!\mid x_0, a_0, \ldots, a_{n-1}, x_n \big) = 1, \quad \forall \, (x_0, a_0, \ldots, a_{n-1}, x_n) \in (\X \times \A)^n \times \X.$$
Such a policy is called \emph{nonrandomized} if in the above every measure on $\A$ is a Dirac measure, and it is called \emph{stationary} if the function $(x_0, a_0, \ldots, a_{n-1}, x_n) \mapsto  \mu_n(d a_n \!\mid\! x_0,  a_0, \ldots, a_{n-1}, x_n)$ depends only on the state $x_n$, in the same way for every $n \geq 0$. In the stationary case, we can write the policy as $\pi = (\mu, \mu, \ldots)$ for a Borel measurable stochastic kernel $\mu(da \,|\, x)$ on $\A$ given $\X$ that obeys the control constraint of the MDP, and we will simply designate this policy by $\mu$. 

Let $\Pi$ denote the space of Borel measurable policies, and let $\Pi_s$ be the subset of all stationary policies in $\Pi$. Given that the action space $\A$ is countable, $\Pi$ and $\Pi_s$ are nonempty (see e.g., \cite[Sect.~2]{Yu19-minp}), and the Borel measurable policies will be adequate for our purpose---henceforth, we shall simply call them policies. We also note that although $\A$ is countable, in the above and throughout the paper, we write probability measures on $\A$ using the general notation for probability measures on a possibly uncountably infinite space, for notational simplicity.

\subsubsection{Average Cost Criterion and Minimum Pair}

In an MDP, a policy $\pi \in \Pi$ and an initial (state) distribution $\zeta \in \P(\X)$ induce a stochastic process $\{(x_n, a_n)\}_{n \geq 0}$ on the infinite product of state and action spaces, $(\X \times \A)^\infty$. The probability measure for this process is uniquely determined by the initial distribution $\zeta$, the sequence of stochastic kernels in $\pi$, and the state transition stochastic kernel $q(dy \,|\, x, a)$ \cite[Prop.~7.28]{bs}. We denote this probability measure by $\Pr^\pi_\zeta$ and the corresponding expectation operator by $\E^\pi_\zeta$. 
The long-run expected average cost of the policy $\pi$ for the initial distribution $\zeta$ is defined by
$$ J(\pi,\zeta) : = \limsup_{n \to \infty} \, n^{-1} \E^\pi_\zeta \big[ \, \textstyle{\sum_{k=0}^{n-1} c(x_k, a_k)} \, \big].$$ 
We shall also refer to $J(\pi, \zeta)$ as the average cost of the pair $(\pi, \zeta)$. 
With the minimum pair approach, we consider the average costs of all policy and initial distribution pairs, 
and among these pairs, we are especially interested in the types of pairs defined below.

Let $\rho^*$ be the \emph{minimum average cost} over all policies and initial distributions:
$$ \rho^* : = \inf_{\zeta \in \P(\X)} \inf_{\pi \in \Pi} J(\pi, \zeta).$$ 

%\smallskip
\begin{definition} \rm \label{def-minp}
A pair $(\pi^*, \zeta^*) \in \Pi \times \P(\X)$ with $J(\pi^*, \zeta^*) = \rho^*$ is called a \emph{minimum pair}. 
\end{definition}
%\smallskip

{\samepage
\begin{definition}[stationary pair and stationary minimum pair] \rm \hfill \label{def-sp}
\begin{enumerate}[leftmargin=0.65cm,labelwidth=!]
\item[(a)] For a stationary policy $\mu \in \Pi_s$ and an initial distribution $p \in \P(\X)$, if $p$ is an invariant probability measure of the Markov chain induced by $\mu$ on $\X$, we call $(\mu, p)$ a \emph{stationary pair}. 
The set of all stationary pairs is denoted by $\Delta_{s}$.
\item[(b)] If $(\mu^*, p^*) \in \Delta_s$ is a minimum pair, we call it a \emph{stationary minimum pair}.
\end{enumerate}
\end{definition}}
%\smallskip

\begin{rem} \rm 
Various terminologies are used in the literature for what we call a stationary pair $(\mu,p)$. In the references \cite{HGL03,HL94,HL99}, the policy $\mu$ is called a ``stable policy'' if $J(\mu, p) < \infty$. In the reference \cite{Bor94}, the probability measure $\gamma(d(x,a)) = \mu(da \,|\, x) \, p(dx)$ is called an ``ergodic occupation measure''---we will discuss such measures in Section~\ref{sec-3.1}.
\myqed
\end{rem}

\subsubsection{Model Assumptions and Existence of Stationary Minimum Pair}

We now impose additional conditions on the MDP model. For a set $B$ in some space, let $B^c$ denote its complement; for a set $B \subset \X \times \A$, let $\proj_\X(B)$ denote the projection of $B$ on $\X$. Recall that $\Gamma = \{(x, a) \mid x \in \X, a \in A(x)\}$. 

\begin{assumption} \label{cond-pc-3} \hfill
\begin{enumerate}[leftmargin=0.85cm]
\item[\rm (G)] For some $\pi \in \Pi$ and $\zeta \in \P(\X)$, the average cost $J(\pi,\zeta) < \infty$.
\item[\rm (SU)] There exists a nondecreasing sequence of compact sets $\Gamma_j \uparrow \Gamma$ such that
$$ \lim_{j \to \infty} \inf_{(x, a) \in \Gamma_j^c} c(x, a) = + \infty.$$
\item[\rm (M)] For each compact set $K \in \{\proj_\X(\Gamma_j)\}$, there exist an open set $O \supset K$, a closed set $D \subset \X$, and a finite measure $\nu$ on $\B(\X)$ (all of which can depend on $K$) such that
\begin{equation}
     q\big((O \setminus D) \cap B  \mid x, a \big) \leq \nu(B), \qquad \forall \, B \in \B(\X), \ (x, a) \in \Gamma, \label{eq-maj-minpair}
\end{equation}     
where the closed set $D$ (possibly empty) is such that restricted to $D \times \A$, the state transition stochastic kernel $q(dy\,|\, x, a)$ is continuous and the one-stage cost function $c(x, a)$ is lower semicontinuous.%footnote starts
\footnote{Since $\A$ is discrete, the continuity condition here means that for each action $a \in \A$, $q(dy\,|\, \cdot, a)$ and $c(\cdot,a)$ are continuous and lower semicontinuous, respectively, on the set $D$.}
%footnote ends 
\end{enumerate}
\end{assumption}
%\smallskip

The first two conditions in this assumption are standard: (G) excludes vacuous problems, and (SU) defines the case of \emph{strictly unbounded} one-stage costs. They were used in, e.g., \cite{HLe93,HL99,Las99,VAm99} to derive average-cost optimality and LP duality results for lower-semicontinuous MDP models with strictly unbounded costs. 

When the function $c(\cdot)$ is lower semicontinuous, (SU) is equivalent to $c(\cdot)$ being \emph{inf-compact} on $\Gamma$, i.e., $E_r : = \{(x,a) \in \Gamma \mid c(x,a) \leq r \}$ is compact for all $r \geq 0$. In our case, these sets $E_r$ need not be closed and instead, (SU) is equivalent to $E_r$ having compact closures. 
Note also that the set $\Gamma$ is $\sigma$-compact under (SU) and, since $\proj_\X(\Gamma_j) \uparrow \X$, the space $\X$ thus must also be $\sigma$-compact.

Condition (M) was introduced in our recent work~\cite{Yu19-minp}. We use the majorization property required in (M) instead of the lower-semicontinuity model conditions commonly required in the literature. The set $D$ in (M) is introduced to separate a ``continuous part'' of the model from the rest, in order to sharpen (M), although this condition can also be used with $D = \varnothing$. Condition (M) seems natural for problems where the probability measures $\{ q(\cdot \,|\, x,a) \mid (x,a) \in \Gamma \}$ have densities on $\X \setminus D$ with respect to (w.r.t.)~a common $\sigma$-finite reference measure and those density functions are bounded uniformly from above. For instance, if $\X = \R^n$ and the reference measure is the Lebesgue measure, we can take $\nu$ in (M) to be a multiple of the Lebesgue measure restricted to a bounded open set that contains $K$.
See \cite[Example~3.2 and Remark 3.3]{Yu19-minp} for more specific examples that illustrate situations where (M) is naturally satisfied or cannot be satisfied. 

Under the preceding assumption, the following results are proved in \cite{Yu19-minp} by making use of Lusin's theorem (see \cite[Thm.~3.5]{Yu19-minp} for sample-path and other optimality properties of a stationary minimum pair). They are analogous to the prior results for lower-semicontinuous MDPs \cite{HLe93,HL99,Kur89}, and they will serve as the starting point for the analyses we present in this paper. 

\begin{theorem}[{optimality of stationary pairs \cite[Prop.~3.2, Thm.~3.3]{Yu19-minp}}] \label{thm-2.1} \ \\
Under Assumption~\ref{cond-pc-3}, the following hold:
\begin{itemize}
\item[\rm (i)] For any pair $(\pi, \zeta) \in \Pi \times \P(\X)$ with $J(\pi, \zeta) < \infty$, there exists a stationary pair $(\bar \mu, \bar p) \in \Delta_{s}$ with $J(\bar \mu, \bar p) \leq J(\pi, \zeta)$.
\item[\rm (ii)] There exists a stationary minimum pair $(\mu^*, p^*) \in \Delta_{s}$.
\end{itemize}
\end{theorem}

\subsection{Linear Programs in Topological Vector Spaces} \label{sec-2.2}
We now give a brief overview of topological vector spaces over the real field and infinite-dimensional linear programs in such spaces. The reader is referred to the books \cite{AnN87,RoR73} for in-depth studies of these subjects, and to the book \cite[Chap.~12.2]{HL99} for a more detailed introduction than ours. Here we shall focus on a few basic concepts and results that will be needed in this paper. 

Let $X$ and $Y$ be two (real) vector spaces, and let $\m0$ denote the element zero for both spaces. 
The pair $(X, Y)$ is called a \emph{dual pair} if there is a bilinear form $\langle \cdot, \, \cdot\rangle: X \times Y \to \R$ such that
\begin{itemize}[leftmargin=0.7cm,labelwidth=!]
\item for each $x \not=\m0$ in $X$, there exists some $y \in Y$ with $\langle x, \, y \rangle \not = 0$,
\item for each $y \not=\m0$ in $Y$, there exists some $x \in X$ with $\langle x, \, y \rangle \not = 0$.
\end{itemize}
For a dual pair $(X, Y)$, the coarsest topology on $X$ under which the function $|\langle \cdot, \, y\rangle|$ is continuous for every $y \in Y$ is called the \emph{weak topology} on $X$ determined by $Y$, and denoted by $\sigma(X, Y)$. By symmetry, $(Y, X)$ is also a dual pair and $\sigma(Y, X)$, the weak topology on $Y$ determined by $X$, is likewise defined.

We recall that a \emph{topological vector space} is a vector space with a topology that is compatible with its algebraic structure (namely, with that topology, the addition and multiplication operations are continuous; see \cite[Chap.~I.3]{RoR73}). 
When endowed with the weak topologies given above, each space in a dual pair $(X, Y)$ is a topological vector space that is separated (i.e., a Hausdorff space) and locally convex (i.e., every point in the space has a base of convex neighborhoods) \cite[Chap.~II.3]{RoR73}.
Convergence in $X$ under the weak topology $\sigma(X, Y)$ can be characterized as follows: a net $\{x_i\}_{i \in \nI}$ in $X$ converges to $\bar x \in X$ if and only if (iff)
$$ \langle x_i, \, y \rangle \to \langle \bar x, \, y \rangle, \qquad \forall \, y \in Y.$$

We consider equality-constrained linear programs and their dual linear programs in topological vector spaces. The definitions of these programs involve several objects, which we introduce first: 
\begin{itemize}[leftmargin=0.7cm,labelwidth=!]
\item two dual pairs of vector spaces $(X, Y)$ and $(Z, W)$, with each space endowed with its respective weak topology;
\item a linear mapping $\L : X \to Z$ that is required to be \emph{weakly continuous} (i.e., $\L$ is continuous under the topology $\sigma(X, Y)$ for $X$ and the topology $\sigma(Z, W)$ for $Z$);
\item a convex cone $\Lambda$ in $X$ and its \emph{dual cone} $\Lambda^*$ in $Y$ defined as
$$ \Lambda^* : = \big\{ y \in Y \mid \langle x, \, y \rangle \geq 0, \ \forall \, x \in \Lambda \big\}.$$
\end{itemize}
The convex cones $\Lambda$ and $\Lambda^*$ induce a partial ordering ``$\leq$" on $X$ and $Y$, respectively:
$$ x_1 \leq x_2 \ \ \text{iff} \ \ x_2 - x_1  \in \Lambda; \qquad   y_1 \leq y_2 \ \ \text{iff} \ \ y_2 - y_1  \in \Lambda^*. $$

The linear mapping $L$ appears in the constraints of a linear program designated as the primal program (P). Associated with $\L$ is another linear mapping $\L^*$ on the space $W$, called the \emph{adjoint} or \emph{transpose} of $\L$, that maps each $w \in W$ to a linear form $\L^*w$ on $X$ and is defined by the identity relation (where $\langle x, \L^* w \rangle$ stands for $(\L^*w)(x)$):
$$ \langle x, \L^* w \rangle : = \langle \L x, w \rangle, \quad \forall \, x \in X, \ w \in W.$$
An important property of $\L$ and $\L^*$ is given by the following proposition:

\begin{prop}[{\cite[Chap.\ II, Prop.~12 and its corollary]{RoR73}}] \label{prp-wcontinuity}
A linear mapping $\L: X \to Z$ is weakly continuous if and only if $\L^*(W) \subset Y$. If $\L$ is weakly continuous, so is $\L^*$.
\end{prop}

This proposition gives a convenient way to verify whether a linear mapping is weakly continuous or not.
When $\L$ is weakly continuous, with the weakly continuous mapping $\L^*: W \to Y$, one can define the dual of the primal linear program.

Let $c \in Y$ and $b \in Z$. Consider the following equality-constrained primal linear program (P) in the space $X$ and its dual linear program (\Pstar) in the space $W$ (cf.~\cite[Chap.~3.3]{AnN87}):
\begin{align}
   \text{(P)} \quad \qquad \text{minimize}  \ \ \ &  \langle \, x,  c \, \rangle \notag \\
                 \text{subject to} \ \ \ & \L x = b, \quad  x \in \Lambda.  \qquad \qquad \label{eq-P0} \\
   \text{(\Pstar)} \quad \ \quad \text{maximize}  \ \ \ &  \langle \, b, w  \,\rangle \notag \\
                 \text{subject to} \ \ \ & - \L^*w + c \in \Lambda^*, \quad w \in W. \label{eq-dP0}
\end{align} 
Similarities between these programs and standard finite-dimensional linear programs can be seen by writing the constraints $x \in \Lambda$ and $- \L^*w + c \in \Lambda^*$ equivalently as $x \geq \m0$ and $\L^* w \leq c$, respectively. 

If the program (P) or (\Pstar) has a feasible solution, it is said to be \emph{consistent}; if it admits an optimal solution, it is said to be \emph{solvable}. 
Let $\inf(\text{P})$ and $\sup  (\Pstar)$ denote the values of (P) and (\Pstar), respectively.
The elementary duality theory (cf.\ \cite[Chap.\ 3.3]{AnN87}) asserts that if (P) and (\Pstar) are both consistent, then 
$$ \sup (\Pstar) \leq \inf (\text{P}).$$
If the equality $\sup (\Pstar) = \inf (\text{P})$ holds, we say there is \emph{no duality gap}. 

There are several sufficient conditions for the absence of a duality gap. 
For our purpose, one duality theorem---Theorem~\ref{thm-subc} below from \cite{AnN87}---will be the most important. 
It characterizes the relation between the value of (\Pstar) and the subvalue of (P), which is defined as follows.

Consider the set $H \subset Z \times \R$ defined by
\begin{equation} \label{eq-H}
  H : = \big\{ \big(\L x, \, \langle \, x, c \, \rangle + r \big) \, \big| \, x \in \Lambda, \ r \geq 0 \big\}.
\end{equation} 
Let $\overbar{H}$ denote the closure of $H$ in the weak topology $\sigma(Z \times \R, \, W \times \R)$ (corresponding to the dual pair $(Z \times \R, \, W \times \R)$ with the bilinear form $\langle (z,r), \, (w, r') \rangle = \langle z, \, w \rangle + r r'$).
We call (P) \emph{subconsistent} if there exists some $r \in \R$ with $(b, r) \in \overbar{H}$.
When (P) is subconsistent, the \emph{subvalue} of (P) is defined by
$$ \text{subvalue} (\text{P}) : = \inf \big\{ r \, \big| \, \big( b, r \big) \in \overbar{H} \, \big\}.$$
For comparison, note that $\inf (\text{P}) = \inf \big\{ r \, \big| \, \big( b, r \big) \in H  \big\}$.
Note also that if $\urho$ is the subvalue of (P), then by the definition of the closure $\overbar{H}$, 
$\big(b,  \urho \big) \in \overbar{H}$ and there exists some net $\{x_i\}_{i \in \nI}$ with $x_i \in \Lambda$ for all $i$, such that $L x_i \to b$ and $\langle x_i, \, c \rangle \to \urho$, where $x_i$ need not be feasible for (P).

\begin{theorem}[subconsistency and duality~{\cite[Thm.~3.3]{AnN87}}] \label{thm-subc}
(P) is subconsistent with a finite subvalue $\urho$ if and only if (\Pstar) is consistent with a finite value $\urho$.
\end{theorem}

%\smallskip
We will apply this theorem in analyzing the duality relationship between the primal and dual linear programs for average-cost MDPs.

\section{Linear Programming for Average-Cost MDPs} \label{sec-3}

In this section we study the LP approach for the average-cost MDP under Assumption~\ref{cond-pc-3}. Roughly speaking, the primal linear program (P) is formulated to find a stationary minimum pair among the stationary pairs of the average cost MDP---this is viable since under Assumption~\ref{cond-pc-3}, the set of stationary pairs is nonempty and a stationary minimum pair exists (cf.\ Theorem~\ref{thm-2.1}). The dual linear program (\Pstar) is then determined by the primal program and the two dual pairs of vector spaces involved in the formulation (cf.\ Section~\ref{sec-2.2}). 
We present the LP formulation and our main duality results in Sections~\ref{sec-3.1} and~\ref{sec-3.2}, respectively. (The proofs of the theorems are given in Section~\ref{sec-5}.) 

Our formulation of the primal linear program is the same as that given by the prior work \cite[Chap.\ 12.3]{HL99}. But our dual program formulation is different; it avoids a condition on the state transition stochastic kernel used in \cite[Chap.\ 12.3]{HL99}, without affecting the desired duality result (cf.\ Remark~\ref{rem-LP}).
This LP formulation we present is one instance of a general class of formulations discussed in the prior work \cite[Sect.~4]{HL94}; however, for the sake of completeness, we will give a detailed account of it using the terminologies introduced in Section~\ref{sec-2.2}.

Regarding notations, in what follows, $\R_+$ denotes the set of nonnegative numbers. For $X = \X$ or $\Gamma$, $\Me(X)$ denotes the space of finite signed Borel measures on $X$, and $\Fn(X)$ the set of real-valued Borel measurable functions on $X$.
We write $\Me^+(X)$ or $\Fn^+(X)$ for the subset of those nonnegative elements in $\Me(X)$ or $\Fn(X)$, and we will use similar notations for the subspaces of $\Me(X)$ or $\Fn(X)$.

For the one-stage cost function $c(\cdot)$, we will also need to work with its restriction to the set $\Gamma$ of state and admissible action pairs (on which $c(\cdot)$ is finite as we recall). For notational simplicity, we shall use the same notation $c$ or $c(\cdot)$ for the restriction of $c(\cdot)$ to $\Gamma$, and the context will make it clear which function is involved in the discussion. 
Likewise, for a Borel measure $\gamma$ on $\Gamma$, sometimes we will also need to work with its extension to the whole state-action space $\X \times \A$, which is simply a Borel measure concentrated on $\Gamma$, and conversely, if $\gamma$ is a Borel measure on $\X \times \A$ concentrated on $\Gamma$, sometimes we will need to consider its restriction to $\Gamma$. In such cases, for notational simplicity, we will use the same notation $\gamma$ for both measures. 

\subsection{Primal and Dual Linear Programs} \label{sec-3.1}

For a Borel measure $\gamma$ on $\Gamma$, let $\hat \gamma$ denote the marginal of $\gamma$ on $\X$.
To define minimization problems on stationary pairs in an MDP, let us first explain a well-known (many-to-one) correspondence between 
a stationary pair $(\mu, p) \in \Delta_s$ and a Borel probability measure $\gamma$ on $\Gamma$ that satisfies 
\begin{equation} \label{eq-inv-gamma}
   \hat \gamma (B) = \int_{\Gamma}  q(B \mid x,a)  \,  \gamma(d(x,a)), \qquad \forall \, B \in \B(\X).
\end{equation}   
The correspondence is essentially given by
\begin{equation} \label{eq-gamma-mup}
\gamma(d(x,a)) = \mu(da \mid x) \, p(dx),
\end{equation}
and has the property that
\begin{equation}  \label{eq-J-gamma}
  J(\mu, p) = \int c \, d\gamma.
\end{equation}  
Indeed, for $(\mu, p) \in \Delta_s$, as $p$ is an invariant probability measure on $\X$ induced by $\mu$, we have
\begin{equation} \label{eq-inv-p}
  p(B) =  \int_{\X} \int_{\A}  q(B \mid x,a)  \, \mu(da \,|\, x) \, p(dx), \qquad \forall \, B \in \B(\X).
\end{equation}  
This is the same as (\ref{eq-inv-gamma}) for the probability measure $\gamma$ given by (\ref{eq-gamma-mup}), since the marginal of $\gamma$ is $\hat \gamma = p$ and $\mu$ obeys the control constraint of the MDP. 
The equality (\ref{eq-J-gamma}) follows from the definition of the average cost and the stationarity of the Markov chain under $\mu$ when the initial distribution is $p$. Conversely, given a probability measure $\gamma$ satisfying (\ref{eq-inv-gamma}), by \cite[Cor.~7.27.2]{bs}, we can decompose $\gamma$ as in (\ref{eq-gamma-mup}) with $p = \hat \gamma$ and $\mu(da \,|\, x)$ being a Borel measurable stochastic kernel  on $\A$ given $\X$ that obeys the control constraint of the MDP.
Then, since $\gamma$ satisfies (\ref{eq-inv-gamma}), the pair $(\mu, p)$ with $p=\hat \gamma$ satisfies (\ref{eq-inv-p}), which means that $p$ is invariant for the Markov chain induced by $\mu$ and hence $(\mu, p)$ is a stationary pair. The policy $\mu$ here is in general not unique; however, by stationarity, every $(\mu, p)$ from this decomposition of $\gamma$ has the same average cost (\ref{eq-J-gamma}).

Due to this correspondence between $(\mu,p)$ and $\gamma$, finding a stationary minimum pair can be expressed as an optimization problem in which one minimizes $\int c \, d\gamma$ over the set of probability measures $\gamma$ that satisfy (\ref{eq-inv-gamma}) (a.k.a.\ the set of ``ergodic occupation measures'' \cite{Bor94}).

Before expressing this optimization problem as a linear program, we also need to restrict attention to those stationary pairs that have finite average costs, so that $\infty$ does not appear in the objective function and the constraints. 
The following definitions are introduced for this purpose.
Consider a positive weight function $w: \Gamma \to \R_+$,
$$ w(x,a) : = 1 + c(x,a), \qquad (x,a) \in \Gamma.$$ 
Let
$\Me_w(\Gamma)$ be the set of finite, signed Borel measures on $\Gamma$ w.r.t.\ which the function $w$ is integrable:
$$ \Me_w(\Gamma): = \textstyle{ \big\{ \gamma \in \Me(\Gamma)  \,\big|\,  \int w \, d |\gamma| < \infty \big\}}$$
where $|\gamma|$ denotes the total variation of $\gamma$.
Let $\Fn_w(\Gamma)$ be the set of Borel measurable functions $\phi$ on $\Gamma$ such that 
$$ | \phi | \leq \ell \, w \quad \text{for some} \  \ell > 0.$$ 
Then every $\phi \in \Fn_w(\Gamma)$ is integrable w.r.t.\ all $\gamma \in \Me_w(\Gamma)$. By (\ref{eq-J-gamma}) and the definition of $w(\cdot)$, if a stationary pair $(\mu, p)$ has finite average cost, then the corresponding probability measure $\gamma \in \Me_w(\Gamma)$. 

We are now ready to define the primal and dual linear programs for the average-cost MDP. Let us specialize the programs (P) and (\Pstar) defined in Section~\ref{sec-2.2}, by identifying the objects involved in those programs as follows:
\begin{itemize}[leftmargin=0.7cm,labelwidth=!]
\item The dual pair $(X, Y) = \big(\Me_w(\Gamma), \Fn_w(\Gamma)\big)$, with the bilinear form 
$$   \langle \gamma \,,\, \phi \rangle : = \int_\Gamma \phi \, d \gamma, \qquad \gamma \in \Me_w(\Gamma), \ \phi \in \Fn_w(\Gamma).$$
\item The dual pair $(Z, W) = \big(\R \times \Me(\X), \, \R \times \Fn_b(\X) \big)$, where
$\Me(\X)$ is the set of finite signed Borel measures on $\X$ as defined earlier, $\Fn_b(\X)$ is the set of bounded Borel measurable functions on $\X$, and the bilinear form on $\big( \R \times \Me(\X) \big) \times \big( \R \times \Fn_b(\X) \big)$  is defined as 
$$ \big\langle (r, \zeta) \,,\, (\rho, h) \big\rangle : = r \rho + \int_\X h \, d \zeta, \qquad (r, \zeta) \in \R \times \Me(\X), \  (\rho, h) \in \R \times \Fn_b(\X).$$
\item The convex cone $\Lambda = \Me_w^+(\Gamma)$, the subset of nonnegative measures in $\Me_w(\Gamma)$. The dual cone of $\Lambda$ is $\Lambda^* = \Fn^+_w(\Gamma)$, the subset of nonnegative functions in $\Fn_w(\Gamma)$.
\item The objective function of the primal program (P) is $\langle \, \gamma, c \, \rangle$, and 
the feasible set of (P) is defined by the following constraints:
\begin{equation} \label{eq-P-feas}
  \gamma \in \Me_w^+(\Gamma), \qquad \gamma(\Gamma) = 1, \qquad \hat \gamma (B) = \int_{\Gamma}  q(B \mid x,a)  \,  \gamma(d(x,a)), \quad \forall \, B \in \B(\X),
\end{equation}  
where $\hat \gamma$ is the marginal of $\gamma$ on $\X$, as we recall.
In other words, in accordance with the earlier discussion, the feasible solutions of (P) correspond to those stationary pairs with finite average costs, and the objective is to minimize the average cost over them.
In the form of (P) discussed in Section~\ref{sec-2.2}, the two equality constraints in (\ref{eq-P-feas}) can be written as 
$$ \L \gamma  = b : = (1, \m0).$$
Here $\m0$ is the trivial measure on $\X$ (i.e., $\m0(B) \equiv 0$ for all $B \in \B(\X)$), and 
the linear mapping $\L$ is defined as $\L:  \Me_w(\Gamma) \to \R \times \Me(\X)$ with $\L = (\L_0, \L_1)$ where, for $\gamma \in \Me_w(\Gamma)$,
\begin{align} 
  \L_0 \gamma  & : = \gamma(\Gamma),  \\
  (\L_1 \gamma)(B) & : = \hat \gamma (B) - \int_{\Gamma}  q(B \mid x,a)  \,  \gamma(d(x,a)), \qquad \forall \, B \in \B(\X).
\end{align} 
\item From the identity $\big\langle \gamma, L^*(\rho, h) \big\rangle = \big\langle L \gamma, (\rho, h) \big\rangle$, the adjoint $\L^*$ of $L$ is given by the linear mapping that maps each $(\rho, h) \in \R \times \Fn_b(\X)$ to the function
\begin{equation}
   \L^*(\rho, h)(x,a) : = \rho + h(x) - \int_\X h(y)  \, q(dy \mid x, a), \qquad (x, a) \in \Gamma.
\end{equation}
Since $\L^*\big(\R \times \Fn_b(\X)\big) \subset \Fn_w(\Gamma)$, both $L$ and $L^*$ are weakly continuous (\cite[Chap.\ II, Prop.~12 and its corollary]{RoR73}; see also Prop.~\ref{prp-wcontinuity}). The inequality constraint in the program (\Pstar) is 
$$- \L^*(\rho, h) + c \in \Fn^+_w(\Gamma).$$ 
We can write this constraint as $\L^*(\rho, h) \leq c$ or more explicitly, as
\begin{equation}
  \rho + h(x) - \int_\X h(y)  \, q(dy \mid x, a) \leq c(x,a), \qquad \forall \, (x, a) \in \Gamma.
\end{equation}
The objective function of the dual program (\Pstar) is $\big\langle b, \, (\rho, h) \big\rangle = \big\langle (1, \m0) , \, (\rho, h) \big\rangle  = \rho$.
\end{itemize}
\smallskip

Expressed in the form introduced in Section~\ref{sec-2.2}, the primal and dual linear programs for the average-cost MDP are:
\begin{align}
   \text{(P)} \quad \qquad \text{minimize}  \ \ \ &  \langle \, \gamma, c \, \rangle \notag \\
                 \text{subject to} \ \ \ & \L\gamma = (1, \m0), \quad  \gamma \in \Me^+_w(\Gamma),   \qquad \ \ \label{eq-P}
\end{align} 
and
\begin{align}
   \text{(\Pstar)} \quad \qquad \text{maximize}  \ \ \ &  \rho \notag \\
                 \text{subject to} \ \ \ & \L^*(\rho, h) \leq c, \quad \rho \in  \R, \ h \in \Fn_b(\X). \label{eq-dP}
\end{align} 
As mentioned earlier, our formulation of (\Pstar) is different from the one given in the book \cite[Chap.~12.3]{HL99}. We will explain the difference and the reason for it in detail in the next subsection (see Remark~\ref{rem-LP}).

A few properties of (P) and (\Pstar) are easy to see. From the relation between stationary pairs and feasible solutions of the primal program (P), it is clear that under Assumption~\ref{cond-pc-3}, the existence of a stationary minimum pair (cf.\ Theorem~\ref{thm-2.1}(ii)) ensures that (P) is both consistent and solvable. Moreover, the proof of Theorem~\ref{thm-2.1}(ii) (cf.\ \cite{Yu19-minp}) shows that due to the strict unbounedness of the one-stage costs, if $\{\gamma_n\}$ is a sequence of feasible solutions of (P) with $\langle \gamma_n, c \rangle \downarrow \inf(\text{P}) = \rho^*$ (such a sequence is called a \emph{minimizing sequence} of (P)), then any subsequence of $\{\gamma_n\}$ contains a further subsequence that converges to an optimal solution of (P) in the topology of weak convergence (of probability measures). The consistency of the dual program (\Pstar) is trivial: since $c \geq 0$, a feasible solution is given by $\rho = 0$ and $h(\cdot) \equiv 0$. We then have $0 \leq \sup (\Pstar) \leq \inf (\text{P}) = \rho^*$ under Assumption~\ref{cond-pc-3}.

Next, we will address the duality between (P) and (\Pstar). 
We will also examine a connection between (\Pstar) and the ACOE for the MDP, through a \emph{maximizing sequence} of (\Pstar).  Such a sequence is defined as a sequence $\{(\rho_n, h_n)\}$ of feasible solutions of (\Pstar) with the property that $\rho_n \uparrow \sup (\Pstar)$.

\subsection{Optimality Results and Discussion} \label{sec-3.2}

Our main result of this section is the absence of a duality gap stated in part (ii) of the following theorem. It can be compared with the prior result of \cite[Chap.~12.3, Thm.~12.3.4]{HL99} for average-cost lower-semicontinuous MDPs. In our case, without lower-semicontinuity model assumptions, we will use Lusin's theorem together with the majorization property in Assumption~\ref{cond-pc-3}(M) to prove it.

\begin{theorem}[consistency and absence of a duality gap] \label{thm-1}
Under Assumption~\ref{cond-pc-3}, the linear programs (P) and (\Pstar) in (\ref{eq-P})-(\ref{eq-dP}) satisfy the following:
\begin{itemize}
\item[\rm (i)] (P) is consistent and solvable, and (\Pstar) is consistent.
\item[\rm (ii)] There is no duality gap: $\inf (\text{P})  = \sup (\Pstar) = \rho^*$.
\end{itemize}
\end{theorem}
%\smallskip

\begin{rem}[about the proof of Theorem~\ref{thm-1}] \label{rem-thm1-proof} \rm 
Besides the differences in assumptions as mentioned above, another difference between our proof of the absence of a duality gap and the proof given in the prior work \cite[Chap.~12.3C]{HL99} is the following. 
The approach of the latter proof is to show that the set $H$ defined by (\ref{eq-H}) is weakly closed (i.e., $H = \overbar{H}$). This is a sufficient condition for the absence of a duality gap, but it requires one to show that every point of $\overbar{H}$ is in $H$.
Our proof uses the duality between the subvalue $\urho$ of (P) and the value of (\Pstar) asserted in \cite[Thm.~3.3]{AnN87} (cf.\ Theorem~\ref{thm-subc}). With this it suffices to show that a single point of $\overbar{H}$, namely, the point $\big(b, \urho\big) = \big((1,\m0), \urho\big)$, is in $H$. Thus our proof is simpler in this respect.

We can also prove that $H$ is weakly closed under our assumptions. 
This requires some minor changes in the proof arguments used in \cite{Yu19-minp}, which we will also use to prove Theorem~\ref{thm-1} (in particular, we only need to change slightly the finite measures used when applying Lusin's theorem). Nonetheless, it will take some space to explain the details of those changes, and this is another reason that we choose to use the duality theorem \cite[Thm.~3.3]{AnN87} instead in our proof.
\myqed
\end{rem}
%\smallskip

\begin{rem}[comparison with a duality result in \cite{Yam75}] \label{rmk-Yamada} \rm 
For compact Euclidean state and action spaces, Yamada proved an LP duality result \cite[Thm.~3]{Yam75} under certain continuity and ergodicity conditions on the MDP. His continuity conditions are different from the lower-semicontinuous model assumption we mentioned, but they can be related to our model assumptions. So let us explain in more detail how our assumptions and duality result compare with his. 
Among others, Yamada assumed that $c(x,a)$ is continuous in $a$ for each fixed $x$, and $q(dy \,|\, x,a)$ has a density $p(y \,|\, x,a)$ w.r.t.\ the Lebesgue measure, where $p(y \,|\, x,a)$ is continuous in $(y,a)$ for each fixed $x$ \cite[Condition (A2)]{Yam75}. 
In our case, since the action space has the discrete topology, trivially, $c(x,a)$ and $q(dy\,|\, x,a)$ are continuous in $a$ for each fixed $x$, so there are similarities to Yamada's conditions. 
Our majorization condition (M) is, however, entirely different from Yamada's geometric ergodicity condition \cite[Conditions (A1), (A4)]{Yam75}, in which he required the density function $p(y \,|\, x,a)$ to be bounded away from zero uniformly for all $(x,a) \in \Gamma$. Using this condition together with the continuity and other assumptions, he proved the absence of a duality gap \cite[Thm.~3]{Yam75}. Both his conditions and his proof arguments are very different from ours.
\myqed
\end{rem}
%\smallskip

\begin{rem}[about the formulation of (\Pstar) and its solvability] \label{rem-LP} \rm 
In defining (\Pstar), we have chosen the space $\Fn_b(\X)$ of \emph{bounded} Borel measurable functions to form the dual pair with the space $\Me(\X)$ of \emph{finite} Borel measures. With this choice, (\Pstar) is in general not solvable (i.e., an optimal solution may not exist), since the inequality 
$$  \rho^* + h(x) \leq c(x,a) + \int_\X h(y) \, q(dy \mid x,a), \qquad \forall \, (x,a) \in \Gamma,$$
need not admit a bounded solution $h$. 

As mentioned earlier, our LP formulation is only an instance of the class of formulations discussed in~\cite[Sect.~4]{HL94}. 
A different dual program (\Pstar) is studied in~\cite[Chap.\ 12.3]{HL99}. It involves, instead of $\big(\R \times \Me(\X), \R \times \Fn_{b}(\X)\big)$, the dual pair $\big(\R \times \Me_{w_0}(\X), \R \times \Fn_{w_0}(\X)\big)$, 
where the two spaces $\Me_{w_0}(\X)$ and $\Fn_{w_0}(\X)$ are defined similarly to $\Me_w(\Gamma)$ and $\Fn_w(\Gamma)$, respectively: with $w_0(x) : = 1 + \inf_{a \in A(x)} c(x,a)$, $x \in \X$,
$$ \Me_{w_0}(\X) : =  \textstyle{ \big\{ p \in \Me(\X)  \,\big|\,  \int w_0 \, d |p| < \infty \big\}}, \quad \Fn_{w_0}(\X)\ : = \big\{ h \in \Fn(\X) \, \big| \, | h | < \ell \, w_0  \  \text{for some} \  \ell > 0\}.$$
This choice leaves more room for (\Pstar) to admit an optimal solution. However, a disadvantage is that to ensure the weak continuity of the linear mapping $L$, an additional condition on the state transition stochastic kernel is required (cf.~\cite[Chap.\ 12.3A, Assumption 12.3.1]{HL99}): for some constant $k > 0$,
\begin{equation} \label{eq-excond-q}
   \int_\X \inf_{a' \in A(y)}\!c(y,a') \, q(dy \mid x,a) \, \leq \, k \big(1 + c(x,a) \big), \qquad \forall \, (x,a) \in \Gamma.
\end{equation}   
Yet, since the costs are strictly unbounded, this condition (\ref{eq-excond-q}) is neither needed for the existence of a minimum pair, nor needed for the absence of a duality gap between (P) and (\Pstar).

Also, the use of the dual pair $\big(\R \times \Me_{w_0}(\X), \R \times \Fn_{w_0}(\X)\big)$ alone cannot guarantee that (\Pstar) has an optimal solution, for which one would still need to make additional assumptions about the functions $h_n$ in a maximizing sequence $\{(\rho_n, h_n)\}$ for (\Pstar)(cf.~\cite[Chap.~12.4B, Thm.~12.4.2]{HL99}). This makes it less appealing to us to have the dual pair $\big(\R \times \Me_{w_0}(\X), \R \times \Fn_{w_0}(\X)\big)$ with its extra condition (\ref{eq-excond-q}) in the LP formulation.

For these reasons, we have formulated (\Pstar) differently. Accordingly, we treat the result on ACOE given in the next proposition not as the property of a dual optimal solution, which may not exist, but as a potential consequence of the results from the LP approach.
\myqed
\end{rem}
\smallskip

As just noted, the dual program (\Pstar) in our formulation need not admit an optimal solution. However, because there is no duality gap, one can still obtain a version of ACOE for the MDP from a maximizing sequence $\{(\rho_n, h_n)\}$ of (\Pstar), under certain conditions on $\{h_n\}$, using essentially the same arguments as those for \cite[Chap.~12.4B, Thm.~12.4.2(c)]{HL99}. 
We include the result in the proposition below, for the sake of completeness. The first part of its condition is satisfied under Assumption~\ref{cond-pc-3} (Theorem~\ref{thm-1}); the second part of its condition specifies the additional conditions on $\{h_n\}$ we need. The ACOE (\ref{eq-ae-acoe}) in the conclusion holds for ``almost all'' (a.a.) states and in general, it \emph{need not hold for all $x \in \X$} (see e.g., \cite[Example 3.1]{Yu19-minp}).

\begin{prop}[ACOE for $p^*$-a.a.\ states] \label{prp-2}
Let $\{(\rho_n, h_n)\}$ be a maximizing sequence of the dual program (\Pstar), and let $h^* = \limsup_{n \to \infty} h_n$.
Suppose that:
\begin{itemize}
\item[\rm (i)] a stationary minimum pair $(\mu^*, p^*)$ exists and $\inf (\text{P})  = \sup (\Pstar) = \rho^* < + \infty$;
\item[\rm (ii)] the functions $h_n$ satisfy that 
$$ \int_\X | h^* | \, d p^* < + \infty,  \qquad \int_\X \sup_{n \geq 0} | h_n(y) | \, q(dy \mid x, a) < + \infty \ \ \ \forall \, (x,a) \in \Gamma.$$
\end{itemize}
Then $h^*$ is finite everywhere,
\begin{equation} 
 \rho^* + h^*(x) \leq  c(x,a) + \int_\X h^*(y) \, q(dy \mid x, a) \qquad  \forall \, (x,a) \in \Gamma,
\end{equation}
and for $p^*$\!-a.a.\ $x \in \X$,
\begin{align} \label{eq-ae-acoe}
 \rho^* + h^*(x) & = \inf_{a \in A(x)} \left\{ c(x,a) + \int_\X h^*(y) \, q(dy \mid x, a) \right\}  \\
   & =   \int_{A(x)}  \left\{ c(x,a) + \int_\X h^*(y) \, q(dy \mid x, a) \right\}  \mu^*(da \mid x).    \label{eq-ae-acoe2}
\end{align}
\end{prop}

%\smallskip
\begin{rem}\rm \label{rmk-nonrandom-pol}
We discuss briefly a relation between the above ACOE and nonrandomzed stationary optimal policies for the average-cost MDP. Firstly, one can find a subset 
$\hat \X \subset \X$ with $p^*(\hat \X) = 1$ and a Borel measurable function $f: \X \to \A$ with $f(x) \in A(x)$ for all $x \in \X$, such that $\hat X$ is absorbing w.r.t.\ $f$ and $f$ attains the minimum in the ACOE (\ref{eq-ae-acoe}) on $\hat X$:
\begin{equation} \label{eq-ae-acoe3}
 q(\hat \X \, |\, x, f(x)) = 1, \qquad  \rho^* + h^*(x)  = c(x, f(x)) + \int_{\hat \X} h^*(y) \, q(dy \mid x, f(x)), \qquad \forall \, x \in \hat \X.
\end{equation}
More specifically, to find such $\hat \X$ and $f$, consider the set $\X'$ with $p^*(\X')=1$ on which (\ref{eq-ae-acoe})-(\ref{eq-ae-acoe2}) hold, and the Markov chain $\{x_n\}$ induced by the policy $\mu^*$ and the initial distribution $p^*$.
Since $p^*$ is an invariant probability measure of this Markov chain, one can construct a set $\hat \X \subset \X'$ with $p^*(\hat \X) = 1$ that is absorbing under $\mu^*$ (see the proof of \cite[Lem.~2.2.3(c)]{HL03} or \cite[Prop.~4.2.3(ii)]{MeT09}). Next, based on the relations (\ref{eq-ae-acoe})-(\ref{eq-ae-acoe2}) on $\hat \X$, the desired function $f$ can be found: this can be done either directly in the special case of a countable action space we have here, or, more generally, by using the Blackwell and Ryll-Nardzewski selection theorem \cite[Thm.~2]{BlR63} as discussed in \cite[Remark 4.6]{HL94}. 

Secondly, for $\hat \X$ and $f$ satisfying (\ref{eq-ae-acoe3}), one can apply standard arguments to show that under certain conditions, the nonrandomized stationary policy $f$ is average-cost optimal for all initial states $x \in \hat \X$. In particular, if $h^* \geq 0$, it is straightforward to show that $f$ is optimal on $\hat \X$. In more general cases of $h^*$, the optimality of $f$ on $\hat \X$ can be established by imposing further conditions to ensure that for all $x \in \hat \X$, $\E_x^f \big[ | h^*(x_n)| \big] < \infty$ for $n \geq 0$ and  $\liminf_{n \to \infty} n^{-1} \E_x^f \big[ h^*(x_n) \big] \geq 0$. (For derivation details, see e.g., the related discussions in \cite[Sect.~3]{HL94} and \cite[Chap.~5.2]{HL96} on canonical triplets.)
\myqed
\end{rem}

\section{Extension to Constrained Average-Cost MDPs} \label{sec-4}

In this section, we extend our results for an unconstrained average-cost MDP to a constrained one. 
Let the state and action spaces and the state transition stochastic kernel of the MDP be the same as before.
Consider multiple one-stage cost functions on $\X \times \A$: $c_0, c_1, \ldots, c_d$. We assume that these functions are nonnegative and Borel measurable, finite on $\Gamma$, and taking the value $+\infty$ outside $\Gamma$. 
The goal is to minimize the average cost w.r.t.\ $c_0$, while keeping the average costs w.r.t.\ $c_1, \ldots, c_d$ within given limits.

More specifically, let $\kappa : = (\kappa_1, \ldots, \kappa_d) \geq 0$ be prescribed upper limits on the average costs in the constraints. For a policy $\pi$ and initial distribution $\zeta$, let $J_i(\pi, \zeta)$ denote the average cost of this pair w.r.t.\ $c_i$, $i = 0, 1, \ldots, d$. 
Define the feasible set of policy and initial distribution pairs by
\begin{equation}
 \S : = \big\{  (\pi, \zeta) \in   \Pi \times \P(\X) \, \big| \, J_0(\pi, \zeta) < \infty, \ J_i(\pi, \zeta) \leq \kappa_i, \, i = 1, \ldots, d \, \big\}.
\end{equation} 
Define the optimal average cost of this constrained problem to be
$$ \rho_c^* : = \inf_{(\pi, \zeta) \in \S} J_0(\pi, \zeta).$$

As before, within the feasible set $\S$, we are especially interested in those stationary pairs. Analogous to the minimum pairs and stationary minimum pairs for an unconstrained MDP, let us define optimal pairs and stationary optimal pairs for the constrained MDP. (What we call optimal pairs are called ``constrained optimal pairs'' in the prior work \cite{KNH00}.)

%\smallskip
\begin{definition}[optimal pairs] \rm \label{def-cmdp-minp} \hfill
\begin{enumerate}[leftmargin=0.65cm,labelwidth=!]
\item[(a)] We call $(\pi^*, \zeta^*) \in \Pi \times \P(\X)$ an \emph{optimal pair} for the constrained MDP if 
$$ (\pi^*, \zeta^*) \in \S \quad \text{and} \quad J_0(\pi^*, \zeta^*)  = \rho^*_c.$$
\item[(b)]
We call an optimal pair $(\pi^*, \zeta^*)$ \emph{lexicographically optimal} if for each $(\pi, \zeta) \in \S$, either $J_i(\pi^*, \zeta^*) = J_i(\pi, \zeta)$ for all $ 0 \leq i \leq d$, or for some $\bar d \leq d$,
$$  J_i(\pi^*, \zeta^*) = J_i(\pi, \zeta) \ \ \  \forall \, i \leq \bar d -1, \qquad  J_{\bar d}(\pi^*, \zeta^*) < J_{\bar d}(\pi, \zeta). \qquad \quad$$
\end{enumerate}
\end{definition}
%\smallskip

\begin{definition}[stationary optimal pairs] \rm \label{def-cmdp-sminp}
If a stationary pair $(\mu^*, p^*) \in \Delta_s$ is (lexicographically) optimal for the constrained MDP, we call it a \emph{stationary (lexicographically) optimal pair}.
\end{definition}
%\smallskip

In what follows, we first adapt the strict unboundedness condition (SU) and the majorization condition (M) to accommodate multiple one-stage cost functions in the constrained MDP, and under those modified conditions we show that stationary optimal pairs exist (Section~\ref{sec-4.1}). We then formulate primal/dual linear programs for the constrained MDP and present duality results that are analogous to the ones for unconstrained problems (Section~\ref{sec-4.2}). The proofs of the theorems of this section are collected in Section~\ref{sec-4.3}.

\subsection{Model Assumptions and Existence of Stationary Optimal Pairs} \label{sec-4.1}

We impose the following conditions on the constrained MDP model:

\begin{assumption} \label{cond-su-cmdp} \hfill
\begin{enumerate}[leftmargin=0.85cm]
\item[\rm (G)] The feasible set $\S \not=\varnothing$.
\item[\rm (SU)] There exists a nondecreasing sequence of compact sets $\Gamma_j \uparrow \Gamma$ such that for some $0 \leq i \leq d$,
$$ \lim_{j \to \infty} \inf_{(x, a) \in \Gamma_j^c} c_i(x, a) = + \infty.$$
\item[\rm (M)] For each compact set $K \in \{\proj_\X(\Gamma_j)\}$, there exist an open set $O \supset K$, a closed set $D \subset \X$, and a finite measure $\nu$ on $\B(\X)$ (all of which can depend on $K$) such that
\begin{equation}
     q\big((O \setminus D) \cap B  \mid x, a \big) \leq \nu(B), \qquad \forall \, B \in \B(\X), \ (x, a) \in \Gamma, \notag 
\end{equation}     
where the closed set $D$ (possibly empty) is such that restricted to $D \times \A$, the state transition stochastic kernel $q(dy\,|\, x, a)$ is continuous and all the one-stage cost functions $c_i$, $0 \leq i \leq d$, are lower semicontinuous. 
\end{enumerate}
\end{assumption}
%\smallskip

This assumption is similar to Assumption~\ref{cond-pc-3} for the unconstrained problem. Condition (G) is to exclude vacuous problems. Condition (SU) is the same as that considered in \cite{HGL03} for the constrained MDP, and it differs from Assumption~\ref{cond-pc-3}(SU) in that here we require \emph{some} one-stage cost function in the constrained problem to be strictly unbounded.
Condition (M) is almost identical to Assumption~\ref{cond-pc-3}(M) except that here the closed set $D$ must be such that on it, \emph{every} one-stage cost function in the constrained problem is lower semicontinuous in the state variable. As before, having a nonempty set $D$ in the majorization condition (M) sharpens this condition by allowing us to treat a ``continuous'' part of the model separately from the rest.

Theorem~\ref{thm-4.1} below extends our earlier results for MDPs \cite[Prop.\ 3.2,~Thm.~3.3]{Yu19-minp} (cf.\ Theorem~\ref{thm-2.1}) to constrained MDPs. 
In particular, its part (i) can be compared with Theorem~\ref{thm-2.1}(i), and its parts (ii)-(iii) with Theorem~\ref{thm-2.1}(ii).
The proof will only be outlined in Section~\ref{sec-4.3}, as it is mostly based on the arguments given in \cite{Yu19-minp}---roughly speaking, the present majorization condition allows us to apply the reasoning in \cite{Yu19-minp} to every one-stage cost function $c_i$ in the constrained MDP.

Parts (i)-(ii) of this theorem are also comparable with the results of \cite[Thm.~3.2]{HGL03} and \cite[the solvability part of Lem.~2.3]{KNH00} for constrained lower-semicontinuous MDPs. 
Part (iii) concerns lexicographically optimal solutions of the constrained MDP, which can be related to solutions for multi-objective MDPs similar to those discussed in \cite{HR04}.

{\samepage
\begin{theorem}[optimality of stationary pairs] \label{thm-4.1}
Under Assumption~\ref{cond-su-cmdp}, the following hold:
\begin{itemize}
\item[\rm (i)] For any pair $(\pi, \zeta) \in \S$, there exists a stationary pair $(\bar \mu, \bar p) \in \Delta_{s} \cap \S$ with 
$$J_i(\bar \mu, \bar p) \leq J_i(\pi, \zeta),  \qquad \forall \, i = 0, \ldots, d.$$
\item[\rm (ii)] There exists a stationary optimal pair $(\mu^*, p^*) \in \Delta_{s} \cap \S$.
\item[\rm (iii)] There exists a stationary lexicographically optimal pair $(\mu^*, p^*) \in \Delta_{s} \cap \S$.
\end{itemize}
\end{theorem}}
%\smallskip

\begin{rem} \rm 
It is known that even in a finite-state-and-action MDP, for a given initial state or distribution, there need not exist a stationary optimal policy for the constrained average cost problem. See \cite[Sect.~4, p.~284]{HoK84} for an interesting counterexample (involving a multichain MDP) that is due to Derman~\cite{Der63}. The difference between this known fact and the existence of a stationary optimal pair in Theorem~\ref{thm-4.1} is that in the constrained MDP here, the initial distribution is not given and there is freedom of choosing it to optimize the average costs.
\myqed
\end{rem}
%\smallskip

\begin{rem}[pathwise average costs of $\mu^*$] \rm 
Suppose that in part (ii) or (iii) of Theorem~\ref{thm-4.1}, the policy $\mu^*$ induces on $\X$ a \emph{positive Harris recurrent} Markov chain (see e.g., \cite[Chap.\ 10.1]{MeT09} for definition). Then, by the ergodic properties of such Markov chains and by the same proof of \cite[Thm.~3.5(b)]{Yu19-minp}, we have that for all initial distributions $\zeta$, $\Pr_\zeta^{\mu^*}$-almost surely, 
$$ \lim_{n \to \infty} n^{-1} \textstyle{\sum_{k=0}^{n-1}} \, c_i(x_k, a_k) = J_i(\mu^*, p^*),  \qquad i = 0, 1, \ldots, d.$$
In other words, almost surely, on each sample path, the pathwise average costs of the policy $\mu^*$ w.r.t.\ $c_i$, $i = 1, 2, \ldots d$, are also within the prescribed limits $\kappa_i$, while its pathwise average cost w.r.t.\ $c_0$ equals $\rho_c^*$ as well.
\myqed
\end{rem}

\subsection{Linear Programming Formulation and Optimality Results} \label{sec-4.2}

Similarly to the unconstrained case, for the constrained MDP, the primal linear program (P) is formulated to minimize the average cost $J_0(\pi, \zeta)$ over feasible stationary pairs, by utilizing the correspondence between a stationary pair and a probability measure that satisfies (\ref{eq-inv-gamma}) discussed at the beginning of Section~\ref{sec-3.2}. Under Assumption~\ref{cond-su-cmdp}, the existence of a stationary optimal pair given by Theorem~\ref{thm-4.1} ensures that such a pair can be obtained by solving the primal program (P).
The dual linear program (\Pstar) is, as before, determined by (P) and two dual pairs of vector spaces we choose. 

We now define precisely (P) and (\Pstar) for the constrained MDP, by identifying the spaces and linear mappings involved in the general LP formulation given in Section~\ref{sec-2.2}. 
To define the primal linear program (P), we consider the dual pair of vector spaces 
$$\big(\Me_w(\Gamma) \times \R^d, \, \Fn_{w}(\Gamma) \times \R^d\big)$$ 
where the weight function $w : \Gamma \to \R_+$ is given by
$$w(x,a) = 1 + \sup_{0 \leq i \leq d} c_i(x,a), \qquad (x,a) \in \Gamma.$$
The bilinear form associated with this dual pair is defined as the sum of the bilinear forms associated with the two dual pairs, $\big(\Me_w(\Gamma), \, \Fn_w(\Gamma) \big)$ and $(\R^d, \, \R^d)$; i.e.,
\begin{align} \label{eq-cmdp-bilinear}
\big\langle (\gamma, \alpha) \,,\, (\phi, \alpha') \rangle  & : = \langle \gamma, \, \phi \rangle + \langle \alpha, \, \alpha' \rangle = \int_\Gamma \phi \, d \gamma + \sum_{i=1}^d \alpha_i \alpha_i' 
\end{align}   
for $\gamma \in \Me_w(\Gamma)$, $\phi \in \Fn_w(\Gamma)$, and $\alpha, \alpha' \in \R^d$ (with $\alpha_i, \alpha'_i$ denoting their $i$th components).

The feasible set of (P) corresponds to the subset of stationary pairs that are feasible for the constrained MDP, and it is defined by the following constraints:
\begin{align*}
      \gamma \in \Me_w^+(\Gamma), \qquad \gamma(\Gamma) = 1, \qquad
    \hat \gamma (B)  = \int_{\Gamma}  q(B \mid x,a)  \,  \gamma(d(x,a)), \qquad \forall \, B \in \B(\X),
\end{align*}
and 
$$  \big( \langle \gamma, \, c_1 \rangle, \, \cdots, \, \langle \gamma, \, c_d \rangle \big) + \alpha = \kappa, \qquad \alpha \geq 0.$$
Note that if $\gamma$ is a probability measure associated with some stationary pair $(\mu, p) \in \S$ via (\ref{eq-gamma-mup}), then $\gamma$ is feasible for (P); in particular, $\langle \gamma, w \rangle \leq 1 + \sum_{i=1}^d \langle \gamma, c_i \rangle < \infty$, so $\gamma \in \Me_w^+(\Gamma)$.
The objective of (P) is to minimize the average cost $\langle \gamma, c_0 \rangle$. 
We can state the primal program (P) in the form introduced in Section~\ref{sec-2.2} as follows:
\begin{align}
   \text{(P)} \quad \qquad \text{minimize}  \ \ \ &  \langle \, \gamma, c_0 \, \rangle \notag \\
                 \text{subject to} \ \ \ & \L(\gamma, \alpha) = (1, \m0, \kappa), \quad  \gamma \in \Me^+_w(\Gamma), \ \ \alpha \geq 0  \quad \ \ \label{eq-cmdp-P}
\end{align} 
where the linear mapping $\L:  \Me_w(\Gamma) \times \R^d \to \R \times \Me(\X) \times \R^d$ is given by 
$\L = (\L_0, \L_1, \L_2)$ with 
\begin{align} 
  \L_0 (\gamma,\alpha)  & : = \gamma(\Gamma),  \\
  \L_1 (\gamma, \alpha)(B) & : = \hat \gamma (B) - \int_{\Gamma}  q(B \mid x,a)  \,  \gamma(d(x,a)), \qquad \forall \, B \in \B(\X), \\
  \L_2 (\gamma, \alpha) & : = \big( \langle \gamma, \, c_1 \rangle, \, \cdots, \, \langle \gamma, \, c_d \rangle \big) + \alpha,
\end{align}  
for $\gamma \in \Me_w(\Gamma)$ and $\alpha = (\alpha_1, \ldots, \alpha_d) \in \R^d$.

To define the dual linear program (\Pstar), we consider the dual pair of vector spaces
$$ \big(\R \times \Me(\X) \times \R^d, \, \R \times \Fn_b(\X) \times \R^d \big),$$
with the bilinear form defined as the sum of the bilinear forms for the three dual pairs, $(\R,\, \R)$, $\big(\Me(\X), \, \Fn_b(\X)\big)$, and $(\R^d,\, \R^d)$, similar to (\ref{eq-cmdp-bilinear}). 
From the definition of $\L$, the adjoint mapping $\L^*$ can be identified: it is the linear mapping 
$\L^*= (\L_1^*, \L_2^*)$ on $\R \times \Fn_b(\X) \times \R^d$ given by
\begin{align}
   \L^*_1(\rho, h, \beta) (x,a)& : = \rho + h(x) - \int_\X h(y) \, q(dy \mid x, a) + \sum_{i =1}^d \beta_i c_i(x,a), \qquad (x,a) \in \Gamma,\\
   \L^*_2(\rho, h, \beta) & : = \beta,
\end{align} 
for $(\rho, h, \beta) \in \R \times \Fn_b(\X) \times \R^d$.
Clearly, $L^*(\R \times \Fn_b(\X) \times \R^d) \subset \Fn_w(\Gamma) \times \R^d$, so both linear mappings $\L$ and $\L^*$ are weakly continuous (\cite[Chap.\ II, Prop.~12 and its corollary]{RoR73}; cf.\ Prop.~\ref{prp-wcontinuity}).
The objective function of (\Pstar) is 
$$ \big\langle (1, \m0, \kappa), \, (\rho, h, \beta) \big\rangle = \rho + \sum_{i=1}^d \beta_i \kappa_i.$$ 
Let us now state the dual program (\Pstar) in the form introduced in Section~\ref{sec-2.2}:
\begin{align}
   \text{(\Pstar)} \quad \qquad \text{maximize}  \ \ \ &  \rho + \sum_{i=1}^d \beta_i \kappa_i \notag \\
                 \text{subject to} \ \ \ & \L^*(\rho, h, \beta) \leq (c_0, 0), \quad \rho \in  \R, \ h \in \Fn_b(\X), \ \beta \in \R^d. \label{eq-cmdp-dP}
\end{align} 
Note that the inequality constraint in (\ref{eq-cmdp-dP}) is the same as the cone constraint
$ - \L^*(\rho, h, \beta) + \big( c_0, \, 0 \big) \in \Fn^+_w(\Gamma) \times \R_+^d$
(cf.\ Section~\ref{sec-2.2}), and it can be expressed more explicitly as
\begin{align}
\rho + h(x) - \int_\X h(y) \, q(dy \mid x, a) + \sum_{i=1}^d  \beta_i c_i(x,a) & \ \leq \ c_0(x, a), \qquad \forall \, (x,a) \in \Gamma, \label{eq-cmdp-dual-cstr1}\\
\beta & \ \leq \ 0. \label{eq-cmdp-dual-cstr2}
\end{align} 

The next theorem about the primal/dual programs (P) and (\Pstar) is an extension of Theorem~\ref{thm-1} to the constrained MDP. The solvability of (P) is a consequence of the existence of a stationary optimal pair given in Theorem~\ref{thm-4.1}(ii). Moreover, the proof of Theorem~\ref{thm-4.1}(ii) also shows that any minimizing sequence $\{(\gamma_n, \alpha_n)\}$ of (P) has a subsequence $(\gamma_{n_k}, \alpha_{n_k}) \to (\gamma^*, \alpha^*)$, where $(\gamma^*, \alpha^*)$ is an optimal solution of (P) and $\gamma_{n_k} \to \gamma^*$ in the topology of weak convergence of probability measures.

The absence of a duality gap is the main result of this section. Its proof, outlined in Section~\ref{sec-4.3}, uses essentially the same proof arguments for Theorem~\ref{thm-1}(ii), which handle the discontinuous MDP models by making use of Lusin's theorem together with the majorziation property in Assumption~\ref{cond-su-cmdp}(M).

%\smallskip
\begin{theorem}[consistency and absence of a duality gap] \label{thm-4.2}
Under Assumption~\ref{cond-su-cmdp}, the following hold for the linear programs (P) and (\Pstar) given in (\ref{eq-cmdp-P}) and (\ref{eq-cmdp-dP}):
\begin{itemize}[leftmargin=0.7cm,labelwidth=!]
\item[\rm (i)] (P) is consistent and solvable, and (\Pstar) is consistent.
\item[\rm (ii)] There is no duality gap: $\inf (\text{P})  = \sup (\Pstar) = \rho^*_c$.
\end{itemize}
\end{theorem}
%\smallskip

This theorem is comparable with the prior results \cite[Thm.~4.4]{HGL03} and \cite[Lem.~2.3]{KNH00} on the LP approach for constrained lower-semicontinuous MDPs (\cite{KNH00} considers compact spaces, and \cite{HGL03} non-compact spaces). Besides the differences in model assumptions, our formulation of the dual program (\Pstar) also differs from that in \cite{HGL03}. The main difference lies in the choice of the spaces $\Me(\X)$ and $\Fn_{b}(\X)$ for (\Pstar). As in the unconstrained case, our motivation for this choice is to avoid an extra condition on the state transition stochastic kernel used in \cite{HGL03}, which is the same condition (\ref{eq-excond-q}) from \cite[Chap.\ 12.3]{HL99} that we discussed earlier in Remark~\ref{rem-LP}. For the same reason as explained in Remark~\ref{rem-LP}, the dual program (\Pstar) as we formulated above need not admit an optimal solution.

For completeness, in the rest of this section, we discuss some solution properties of the dual program (\Pstar) and derive a version of ACOE for the constrained MDP. 
Consider a maximizing sequence $\{(\rho_n, h_n, \beta_n)\}$ of (\Pstar), i.e., feasible solutions of (\Pstar) with $\rho_n + \langle \kappa, \, \beta_n \rangle \uparrow \sup (\Pstar)$.
We first examine the boundedness property of $\{\beta_n\}$. Denote by $\beta_{n,j}$ the $j$th component of $\beta_n$.
Let us separate the constraints of the MDP into two categories:
\begin{equation}
 \J^{(0)} : = \big\{ i \mid 1 \leq i \leq d, \ \exists\, (\pi, \zeta) \in \S \ s.t.\ J_i(\pi, \zeta) < \kappa_i \big\}, \qquad \J^{(1)} : = \{ 1, 2, \ldots, d\} \setminus \J^{(0)}. \quad
\end{equation} 
When $\S \not= \varnothing$, $\J^{(1)}$ consists of all those $i$ such that w.r.t.\ $c_i$, every feasible pair in $\S$ has the same maximally allowed average cost $\kappa_i$. 

%\smallskip
\begin{prop}  \label{prp-4.3}
Suppose Assumption~\ref{cond-su-cmdp} hold. Let $\{(\rho_n, h_n, \beta_n)\}$ be a maximizing sequence of the dual program (\Pstar).
Then the following hold:
\begin{itemize}
\item[\rm (i)] The sequence $\{\beta_{n,j}\}_{n \geq 0}$ is bounded for every $j \in \J^{(0)}$.
\item[\rm (ii)] For $1 \leq j \leq d$, $\lim_{n \to \infty} \beta_{n,j} = 0$ if $J_j(\mu^*, p^*) < \kappa_j$ for some stationary optimal pair $(\mu^*, p^*)$ of the constrained MDP. 
\item[\rm (iii)] Suppose there exists $(\pi, \zeta) \in \Pi \times \P(\X)$ such that 
$$  J_j(\pi, \zeta) < \kappa_j \quad \forall \, j \in \J^{(1)}, \qquad J_j(\pi, \zeta) < \infty \quad \forall \, j \in \J^{(0)} \cup \{ 0\}.$$ 
Then the sequence $\{\beta_{n}\}_{n \geq 0}$ is bounded. 
\end{itemize}
\end{prop}
%\smallskip

\begin{rem} \rm 
An optimal solution $(\gamma^*, \alpha^*)$ of (P) corresponds to a stationary optimal pair $(\mu^*, p^*)$ with $\alpha^*_j = \kappa_j - J_j(\mu^*, p^*)$ for $1 \leq j \leq d$ (this follows from the correspondence relationship explained at the beginning of Section~\ref{sec-3.2}). So Prop.~\ref{prp-4.3}(ii) entails the complementarity relation
$\langle \alpha^*, \, \beta^* \rangle = 0$ for an optimal solution $(\gamma^*, \alpha^*)$ of (P), if we define $\beta^* = (\beta^*_1, \ldots, \beta^*_d)$ as follows: $\beta^*_j = \lim_{n \to \infty} \beta_{n,j}$ if this limit exists, and assign $\beta^*_j$ an arbitrary number otherwise.
Proposition~\ref{prp-4.3}(iii) gives a sufficient condition under which the $\J^{(1)}$-components of $\{\beta_n\}$ are also bounded---note that this condition involves \emph{non-feasible} policy and initial distribution pairs and is different from the Slater condition $J_i(\pi, \zeta) < \kappa_i$, $1 \leq i \leq d$. One exceptional case where Prop.~\ref{prp-4.3} is inapplicable is when $\kappa = 0$.
\myqed
\end{rem}
%\medskip

When $\{\beta_{n}\}_{n \geq 0}$ is bounded, as when the condition of Prop.~\ref{prp-4.3}(iii) holds, we can choose a subsequence of the maximizing sequence $\{(\rho_n, h_n, \beta_n)\}$ so that $\beta_n$ converges. The subsequence is obviously also a maximizing sequence for (\Pstar). Then, with additional assumptions on the functions $h_n$, we can derive an optimality equation for the constrained MDP that is analogous to the ACOE (\ref{eq-ae-acoe}) in Prop.~\ref{prp-2} for the unconstrained MDP. We state this result in the next proposition. It is comparable with the result of \cite[Thm.~5.2(b)]{HGL03} for constrained lower-semicontinuous MDPs; in the latter reference, (\ref{eq-cmdp-ae-acoe}) is called the ``constrained optimality equation.''

\begin{prop}[ACOE for $p^*\!$-a.a.\ states in the constrained MDP] \label{prp-4.4} \ \\
Let $\{(\rho_n, h_n, \beta_n)\}$ be a maximizing sequence of the dual program (\Pstar), and let $h^* = \limsup_{n \to \infty} h_n$.
Suppose that:
\begin{itemize}
\item[\rm (i)] a stationary optimal pair $(\mu^*, p^*)$ exists and $\inf (\text{P})  = \sup (\Pstar) = \rho^*_c < + \infty$;
\item[\rm (ii)] the functions $h_n$ satisfy that 
$$ \int_\X | h^* | \, d p^* < + \infty,  \qquad \int_\X \sup_{n \geq 0} | h_n(y) | \, q(dy  \mid x, a) < + \infty \ \ \ \forall \, (x,a) \in \Gamma;$$
\item[\rm (iii)] the sequence $\{\beta_n\}$ converges to some finite $\beta^*$.
\end{itemize}
Then $h^*$ is finite everywhere and with
$$ c^*(x,a) : = c_0(x,a) - \textstyle{\sum_{i=1}^d} \, \beta^*_i c_i(x,a), \qquad \tilde{\rho}^* : = \rho^*_c - \sum_{i=1}^d \beta^*_i \kappa_i,$$
we have
\begin{equation} \label{eq-cmdp-dsol}
  \tilde{\rho}^* + h^*(x) \leq c^*(x,a) +  \int_\X h^*(y) \, q(dy \mid x, a), \qquad \forall \, (x,a) \in \Gamma,
\end{equation}
and for $p^*$\!-a.a.\ $x \in \X$,
\begin{align} \label{eq-cmdp-ae-acoe}
 \tilde{\rho}^* + h^*(x) & = \inf_{a \in A(x)} \left\{ c^*(x,a) + \int_\X h^*(y) \, q(dy \mid x, a) \right\}  \\
   & = \int_{a \in A(x)} \left\{ c^*(x,a) + \int_\X h^*(y) \, q(dy \mid x, a) \right\}  \mu^*(da \mid x). \label{eq-cmdp-ae-acoe2}
\end{align}
\end{prop}
%\smallskip

\section{Proofs} \label{sec-5}
This section collects the proofs of the theorems given in Sections~\ref{sec-3} and~\ref{sec-4}.

\subsection{Proofs for Section~\ref{sec-3}}

Let us first recall a few definitions and facts about probability measures on a metrizable space $X$. 
Let $\C_b(X)$ denote the set of real-valued, bounded continuous functions on $X$. By definition, a sequence of probability measures $p_n \in \P(X)$ \emph{converges weakly} to some $p \in \P(X)$, denoted $p_n \wto p$, if $\int f  d p_n \to \int f  dp$ for all $f \in \C_b(X)$.  
If $\mathcal{E}$ is a family of probability measures in $\P(X)$ such that for any $\epsilon > 0$, there is a compact set $K \subset X$ with $p(K) > 1 - \epsilon$ for all $p \in \mathcal{E}$, we say that $\mathcal{E}$ is \emph{tight}. 

By Prohorov's theorem~\cite[Thm.~6.1]{Bil68}, any sequence in a tight family $\mathcal{E}$ has a further subsequence that converges weakly to a probability measure in $\P(X)$. We will use this fact many times in our proofs, for some family $\mathcal{E} \subset \P(\Gamma)$ that satisfies $\sup_{\gamma \in \mathcal{E}} \langle \gamma, \, c \rangle < \infty$. By the strict unboundedness condition on $c$ given in Assumption~\ref{cond-pc-3}(SU), such a family $\mathcal{E}$ must be tight (as can be seen easily from condition (SU) and the definition of tightness).

\subsubsection{Proof of Theorem~\ref{thm-1}}

The consistency of (P) and (\Pstar) and the solvability of (P) were already discussed in Section~\ref{sec-3.1}, where we also showed that under Assumption~\ref{cond-pc-3}, $0 \leq \sup (\Pstar) \leq \inf (\text{P}) = \rho^*$.

We now prove that there is no duality gap between (P) and (\Pstar). Our approach is to use \cite[Thm.~3.3]{AnN87} (cf.\ Theorem~\ref{thm-subc} in Section~\ref{sec-2.2}), which asserts the equality between the subvalue of (P) and the value of (\Pstar) when they are finite. Specifically, recall from Section~\ref{sec-2.2} that 
the subvalue of (P) is defined as
$$ \urho : = \inf \big\{ r \mid \big( (1, \m0), r \big) \in \overbar{H} \big\},$$
where the set $H \subset \R \times \Me(\X) \times \R$ is given by
\begin{equation}
  H : = \big\{ \big(L \gamma, \, \langle \, \gamma, c \, \rangle + r \big) \, \big| \, \gamma \in \Me_w^+(\Gamma), \ r \geq 0 \big\},
\end{equation} 
and $\overbar{H}$ is the closure of $H$ in the weak topology $\sigma \big(\R \times \Me(\X) \times \R, \, \R \times \Fn_b(\X) \times \R \big)$.
Since (P) and (\Pstar) are consistent, $\sup (\Pstar)$ is finite and equals the subvalue $\urho$ by \cite[Thm.~3.3]{AnN87} (cf.\ Theorem~\ref{thm-subc}). So, to show $\inf(\text{P}) = \sup (\Pstar)$, we need to prove $\rho^* = \urho$. In what follows, we will prove that 
$$\big((1, \m0), \, \urho \, \big) \in H,$$
by constructing a stationary pair whose average cost is no greater than $\urho$. 
This will give us $\rho^* = \urho$ (since it implies $\urho \geq \rho^*$, whereas $\rho^* \geq \urho$). The proof will proceed in four steps, with the first three steps making preparations for the last one.

\medskip
\noindent {\bf Step (i):}
From the definition of $\urho$, it follows that $\big((1, \m0), \urho\big) \in \overbar{H}$ and moreover, there exist a direct set $\nI$ and a net $\{\gamma_i\}_{i \in \nI}$ in $\Me_w^+(\Gamma)$ with 
$$ \big(L \gamma_i \,, \, \langle \, \gamma_i, c \, \rangle \big) \to \big((1, \m0), \urho \,\big)$$
in the $\sigma \big(\R \times \Me(\X) \times \R, \, \R \times \Fn_b(\X) \times \R \big)$ topology. This means that 
\begin{align}
   \gamma_i(\Gamma) & \to 1,   \label{eq-thm1-prf1} \\
  \int_\X h(x) \, \hat \gamma_i(dx) - \int_{\Gamma} \int_\X h(y) \, q(dy \mid x, a) \, \gamma_i(d(x,a)) & \to 0, \qquad \forall \, h \in \Fn_b(\X),  \label{eq-thm1-prf2} \\
   \langle \, \gamma_i, c \, \rangle & \to \urho.  \label{eq-thm1-prf3}
\end{align}
In view of (\ref{eq-thm1-prf1}), there exists $\bar i \in \nI$ such that for all $i \geq \bar i$, $\gamma_i(\Gamma) > 0$. Then, since all $\gamma_i$ are nonnegative measures and $\gamma_i(\Gamma) \to 1$, by restricting attention to $\gamma_i, i \geq \bar i$, and considering the normalized measures $\gamma_i(\cdot)/\gamma_i(\Gamma)$ instead of $\gamma_i$, we can redefine the net $\{\gamma_i\}_{i \in \nI}$ in the above so that every $\gamma_i$ is a probability measure on $\B(\Gamma)$:
$$ \gamma_i(\Gamma) = 1, \qquad \forall \,  i \in \nI.$$

\smallskip
\noindent {\bf Step (ii):} Next, from the net $\{\gamma_i\}_{i \in \nI}$, we will extract a \emph{sequence} of probability measures with the property that the convergence in (\ref{eq-thm1-prf2}) holds for a \emph{countable subset} of the functions in $\Fn_b(\X)$. We start by defining this subset. It consists of two countable families of functions, $\hat \C_b(\X)$ and $\hat \Fn_b(\X)$. The set $\hat \C_b(\X)$ involves continuous bounded functions that will be used to determine if two probability measures on $\X$ are equal. The set $\hat \Fn_b(\X)$ involves indicator functions of certain sets in $\X$ that will be important in the subsequent proof to handle the discontinuities in the MDP model by using Lusin's theorem and the majorization property in Assumption~\ref{cond-pc-3}(M). The construction of $\hat \Fn_b(\X)$ will use the arguments we used in the proof of \cite[Thm.~3.5(a)]{Yu19-minp}. The precise definitions of these two sets are as follows.

Recall that $\C_b(\X)$ is the set of (real-valued) bounded continuous functions on $\X$.
Since $\X$ is metrizable, 
by \cite[Chap.~II, Thm.~6.6]{Par67}, there exists a countable set 
$$\hat \C_b(\X) :  = \{h_1, h_2, \ldots\} \subset \C_b(\X)$$ 
such that in $\P(\X)$, a sequence of probability measures $p_n \wto p \in \P(\X)$ if and only if
$$ \textstyle{ \int h \, d p_n \, \to \, \int h \, dp}, \qquad \forall \, h \in \hat \C_b(\X). $$
Then by \cite[Prop.\ 11.3.2]{Dud02}, for any $p,p' \in \P(\X)$,
\begin{equation} \label{eq-detclass}
  p = p' \qquad \Longleftrightarrow \qquad  \textstyle{\int h \, d p \, = \, \int h \, dp'}, \qquad \forall \, h \in \hat \C_b(\X).
\end{equation}   
The countable set $\hat \C_b(\X)$ is the first family of functions we will need.

We now define the other countable family $\hat \Fn_b(\X)$ of indicator functions mentioned earlier.
The definition of this set involves some new notations and Lusin's theorem. 

Let $\Z_+$ denote the set of all positive integers. 
For $m \in \Z_+$, define the truncated one-stage cost function $c^m(\cdot) : = \min \{c(\cdot), m\}$ on $\X \times \A$ (later, a technical argument in Step (iv) of our proof will involve these $c^m$ functions).
For each $j \in \Z_+$, corresponding to the compact set $\Gamma_j$ in Assumption~\ref{cond-pc-3}(SU), 
let $(O_j, D_j, \nu_j)$ be the open set, the closed set, and the finite measure, respectively, in Assumption~\ref{cond-pc-3}(M) for $K = \proj_\X(\Gamma_j)$.
Let 
$F_j : = \proj_\A(\Gamma_j)$, the projection of $\Gamma_j$ on $\A$. Then the set $F_j$ is compact, and since $\A$ is countable and discrete, this means that the set $F_j$ is finite.

%\smallskip
\begin{lemma} \label{lem-lusin}
For each $j,m \in \Z_+$ and $\ell \in \Z_+$, there exist closed subsets $B^1_{j,m,\ell}$ and $B^2_{j,\ell}$ of $\X$ such that the following hold:
\begin{enumerate}[leftmargin=0.7cm,labelwidth=!]
\item[\rm (i)] $\nu_j \big(\X \setminus B^1_{j,m,\ell} \big) \leq \ell^{-1}$ and $\nu_j \big(\X \setminus B^2_{j,\ell} \big) \leq \ell^{-1}$;
\item[\rm (ii)] restricted to the set $B^1_{j,m,\ell} \times F_j$, the function $c^m(\cdot)$ is continuous, and restricted to the set $B^2_{j,\ell} \times F_j$, the state transition stochastic kernel $q(dy \,|\, \cdot,\cdot)$ is continuous.
\end{enumerate}
\end{lemma}

\begin{proof} 
This lemma is a consequence of Lusin's theorem (see~\cite[Thm.~7.5.2]{Dud02}), which asserts that if $f$ is a Borel measurable function from a topological space $X$ into a separable metric space $S$ and $\nu$ is a closed regular finite Borel measure on $X$, then for any $\delta > 0$, there is a closed set $B$ such that $\nu(X \setminus B) < \delta$ and the restriction of $f$ to $B$ is continuous. 

We apply this theorem with $X = \X$ and $\nu = \nu_j$ for each $j$ in the lemma. 
Since $\X$ is a metrizable topological space, every finite Borel measure is closed regular by \cite[Thm.~7.1.3]{Dud02}, and therefore, the finite measure $\nu_j$ in the lemma meets the condition in Lusin's theorem.

For each $j, m, \ell \in \Z_+$, to find the desired closed set $B^1_{j,m,\ell}$,
we apply Lusin's theorem with $X = \X$, $S = \R$, $\nu = \nu_j$ and $\delta = \ell^{-1}/|F_j|$, and with the function $f(\cdot) = c^m(\cdot, a)$ for each action $a \in F_j$. This gives us, for each $a \in F_j$, a closed set $E_a$ such that $\nu_j(\X \setminus E_a) < \delta$ and restricted to $E_a$, $c^m(\cdot, a)$ is continuous. Then the closed set $B^1_{j,m,\ell} : = \cap_{a \in F_j} E_a$ has the desired property that $\nu_j \big(\X \setminus B^1_{j,m,\ell} \big) \leq \ell^{-1}$ and restricted to $B^1_{j,m,\ell} \times F_j$, $c^m(\cdot, \cdot)$ is continuous.

For each $j, \ell \in \Z_+$, the desired closed set $B^2_{j,\ell}$ is constructed similarly, by applying Lusin's theorem to the state transition stochastic kernel $q(dy \,|\, x, a)$, which is a $\P(\X)$-valued Borel measurable function on $\X \times \A$. Specifically, we let $X = \X$, $S = \P(\X)$, $\nu = \nu_j$, and $\delta = \ell^{-1}/|F_j|$. (Since $\X$ is separable and metrizable, by \cite[Prop.~7.20]{bs}, $\P(\X)$ is also a separable metrizable space and hence meets the condition for the space $S$ in Lusin's theorem.)
We apply Lusin's theorem to $f(\cdot) =  q(dy \,|\, \cdot, a)$ for each $a \in F_j$ to obtain a closed set $E_a$ such that  $\nu_j(\X \setminus E_a) < \delta$ and restricted to $E_a$, $q(dy \,|\, \cdot, a)$ is continuous. We then let the desired set $B^2_{j,\ell} = \cap_{a \in F_j} E_a$.
\end{proof}
%\medskip

We group $(O_j, D_j, \nu_j, B^1_{j,m,\ell})$, $(O_j, D_j, \nu_j, B^2_{j,\ell})$ in the preceding proof into two countable collections $\W_1$ and $\W_2$:
$$ \W_1 := \big\{ (O_j, D_j, \nu_j, B^1_{j,m,\ell}) \mid j,m,\ell \in \Z_+ \big\}, \qquad   \W_2 := \big\{ (O_j, D_j, \nu_j, B^2_{j,\ell}) \mid j,\ell \in \Z_+ \big\}. $$
Let $\ind_E$ denote the indicator function for a set $E$. 
Finally, define a countable set $\hat \Fn_b(\X)$ of indicator functions on $\X$ by
\begin{equation} \label{eq-def-indfn}
    \hat \Fn_b(\X) : = \big\{ \ind_E(\cdot) \mid  E = (O \setminus D) \cap B^c \ \, \text{for some} \ \, (O, D, \nu, B) \in \W_1 \cup \W_2 \big\}.
\end{equation}
Note that the sets $E$ in (\ref{eq-def-indfn}) are open sets (since $O$ is open and $D, B$ are closed); this fact will be useful later.

We now extract a desirable sequence from the net $\{\gamma_i\}_{i \in \nI}$:

%\smallskip
\begin{lemma} \label{lem-net2seq}
There exists a sequence $\{\gamma_n\}_{n \geq 0} \subset \{\gamma_i\}_{i \in \nI}$ such that
\begin{align}
  \int_\X h(x) \, \hat \gamma_n(dx) - \int_{\Gamma} \int_\X h(y) \, q(dy \mid x, a) \, \gamma_n(d(x,a)) & \to 0, \qquad \forall \, h \in \hat \C_b(\X) \cup \hat \Fn_b(\X),  \label{eq-thm1-prf2a} \\
   \langle \, \gamma_n, c \, \rangle  & \to \urho.  \label{eq-thm1-prf3a}
\end{align}
\end{lemma} 

\begin{proof}
Let us order the functions in the countable set $\hat \C_b(\X) \cup \hat \Fn_b(\X)$ as $h_1, h_2, \ldots$. 
Choose any $\bar i_0 \in \nI$ and let $\gamma_n = \gamma_{\bar i_0}$ for $n = 0$. 
For each $n \geq 1$, by (\ref{eq-thm1-prf2})-(\ref{eq-thm1-prf3}), there exists $\bar i_n \in \nI$, $\bar i_n \geq \bar i_{n-1}$ such that for all $i \geq \bar i_n$, 
\begin{align*}
 \left| \, \int_\X h(x) \, \hat \gamma_i(dx) - \int_{\Gamma} \int_\X h(y) \, q(dy \mid x, a) \, \gamma_i(d(x,a)) \, \right| & \leq  n^{-1}, \qquad \forall \,  h \in \big\{ h_1, h_2, \, \ldots, h_n \big\},  \\
 \urho - n^{-1} \leq  \langle \, \gamma_i, c \, \rangle  & \leq \urho+ n^{-1}.  
\end{align*}
Let $\gamma_n = \gamma_{\bar i_n}$.  The resulting sequence $\{\gamma_n\}_{n \geq 0}$ satisfies (\ref{eq-thm1-prf2a})-(\ref{eq-thm1-prf3a}).
\end{proof}

\smallskip
\noindent {\bf Step (iii):} Henceforth, we work with the sequence $\{\gamma_n\}$ of probability measures given by Lemma~\ref{lem-net2seq}.
The relation (\ref{eq-thm1-prf3a}) together with Assumption~\ref{cond-pc-3}(SU) implies that $\{\gamma_n\}$ is a tight family of probability measures on $\B(\Gamma)$. So by Prohorov's theorem \cite[Thm.~6.1]{Bil68}, it has a subsequence that converges weakly to some probability measure $\bar \gamma$ on $\B(\Gamma)$. To simplify notation, let us use the same notation $\{\gamma_n\}$ to denote the convergent subsequence. Thus $\gamma_n \wto \bar \gamma$.

By \cite[Cor.~7.27.2]{bs}, the probability measure $\bar \gamma$ 
can be decomposed into its marginal $\bar p$ on $\X$ and a stochastic kernel $\bar \mu$ on $\A$ given $\X$ that obeys the control constraint of the MDP; i.e.,
$$  \bar \gamma(d(x,a)) = \bar \mu(da \mid x) \, \bar p(dx). $$
This gives us a stationary policy $\bar \mu$. Before we investigate the property of the pair $(\bar \mu, \bar p)$ in the next step, we need the following majorization property, which will be used to deal with the discontinuities in the MDP model:

\begin{lemma} \label{lem-su-mp}
For every $(O, D, \nu, B) \in \W_1 \cup \W_2$,
$$ \limsup_{n \to \infty} \, \hat \gamma_{n} \big((O \setminus D)\cap B^c \big) \leq \nu(B^c), \qquad   \bar p\big( (O \setminus D)\cap B^c \big) \leq \nu(B^c). $$
\end{lemma}

\begin{proof}
For $(O, D, \nu, B) \in \W_1 \cup \W_2$, let $E = (O \setminus D) \cap B^c$ and since the indicator function 
$\ind_E \in \hat \Fn_b(\X)$, we have, by (\ref{eq-thm1-prf2a}) in Lemma~\ref{lem-net2seq}, that
$$ \epsilon_n : = \textstyle{\big| \, \hat \gamma_n(E) - \int_{\Gamma} q(E \mid x, a) \, \gamma_n(d(x,a)) \, \big|} \to 0.$$
We also have, by Assumption~\ref{cond-pc-3}(M),
$$ \int_{\Gamma} q(E \mid x, a) \,  \gamma_n(d(x,a)) \leq \int_{\Gamma} \nu(B^c) \, \gamma_n(d(x,a)) = \nu(B^c).$$
Hence $\hat \gamma_n ( E) \leq \nu(B^c) + \epsilon_n$ for all $n \geq 0$; consequently, $\limsup_{n \to \infty} \!\hat \gamma_{n} (E) \leq \nu(B^c)$.

Now $\hat \gamma_n \wto \bar p$ (since $\gamma_n \wto \bar \gamma$) and $E$ is an open set (since $O$ is open and $D, B$ are closed). Therefore, by \cite[Thm.~11.1.1]{Dud02} and the first part of the proof,
$\bar p(E) \leq \liminf_{n \to \infty} \!\hat \gamma_n (E) \leq \nu(B^c)$.
\end{proof}

\smallskip
\noindent {\bf Step (iv):} We are now ready to prove that $\big((1,\m0), \urho \big) \in H$.

\begin{lemma} \label{lem-sp-subv}
The pair $(\bar \mu, \bar p)$ is a stationary pair with $J(\bar \mu, \bar p) = \langle \bar \gamma, \, c \rangle \leq \urho$. 
\end{lemma}

\begin{proof}[Proof outline] 
We will only outline the proof, because the arguments for this lemma are essentially the same as those we used in an earlier work to prove the existence and pathwise optimality properties of stationary pairs~\cite[Sect.~4.1 and Sect.~4.3.1]{Yu19-minp}. 
By Lemma~\ref{lem-net2seq}, it suffices to prove the inequality
\begin{equation} \label{eq-lemdp-prf1}
   \langle \bar \gamma, \, c \rangle \leq \lim_{n \to \infty}  \langle \, \gamma_n, c \, \rangle  = \urho
\end{equation} 
and to prove that for all $h \in \hat \C_b(\X)$,
\begin{equation} \label{eq-lemdp-prf2}
   \lim_{n \to \infty}  \int_{\Gamma} \int_{\X} h (y) \, q(dy \mid x, a) \, \gamma_{n} (d(x,a)) = \int_{\Gamma} \int_{\X} h(y) \, q(dy \mid x, a) \, \bar \gamma (d(x,a)).
 \end{equation}
To see the sufficiency of (\ref{eq-lemdp-prf1}) and (\ref{eq-lemdp-prf2}), note that (\ref{eq-lemdp-prf2}), together with (\ref{eq-thm1-prf2a}) in Lemma~\ref{lem-net2seq} and the fact 
$\lim_{n \to \infty} \int h \, d \hat \gamma_n =  \int h \, d \bar p$ for all $h \in \hat \C_b(\X)$ (since $\hat \gamma_n \wto \bar p$),
will imply that
$$ \int_{\Gamma} \int_{\X} h(y) \, q(dy \mid x, a) \, \bar \gamma (d(x,a)) = \int_\X h \,  \bar p(dx), \qquad \forall \, h \in \hat \C_b(\X).$$
In turn, this will imply that $\bar p$ is identical to the probability measure $\int_{\Gamma} q(\cdot \mid x, a) \, \bar \gamma (d(x,a))$ (cf.~(\ref{eq-detclass})), thus proving that $\bar p$ is an invariant probability measure for the Markov chain induced by the policy $\bar \mu$ and hence $(\bar \mu, \bar p)$ is a stationary pair. Then the first relation (\ref{eq-lemdp-prf1}) will give us the desired inequality $J(\bar \mu, \bar p) = \langle \bar \gamma, \, c \rangle  \leq \urho$.

\medskip
\noindent {\bf Proving (\ref{eq-lemdp-prf1}):} The proof of (\ref{eq-lemdp-prf1}) is essentially the same as that given in \cite[Sect.~4, proofs of Lems.~4.3, 4.9]{Yu19-minp}. Below, we sketch the main proof arguments (see the proofs in \cite{Yu19-minp} for the details of each step):
\begin{enumerate}[leftmargin=0.7cm,labelwidth=!]
\item[1.] To show (\ref{eq-lemdp-prf1}), it suffices to show that for each $m \in \Z_+$, 
\begin{equation} \label{eq-lemdp-prf3}
   \int c^m  \, d \bar \gamma  \, \leq \, \liminf_{n \to \infty} \int c^m  \, d  \gamma_n.
\end{equation}   
(In the above, the probability measures $\bar \gamma$ and $\gamma_n$ are extended from $\Gamma$ to $\X \times \A$, and $c^m$ is the truncated one-stage cost function $\min \{c(\cdot), m\}$, as we recall.)
\item[2.] Fix $m$. To prove (\ref{eq-lemdp-prf3}), consider arbitrarily small $\epsilon = \delta = \ell^{-1}$, for some arbitrarily large $\ell \in \Z_+$. 
Assumption~\ref{cond-pc-3}(SU) together with (\ref{eq-thm1-prf3a}) in Lemma~\ref{lem-net2seq} allows us to choose $j \in \Z_+$ large enough so that for the compact set $\Gamma_j$ in Assumption~\ref{cond-pc-3}(SU), we have $\gamma_n(\Gamma_j^c) \leq \epsilon$ for all $n$ and $\bar \gamma(\Gamma_j^c) \leq \epsilon$. 
This in turn allows us to bound $\int_{\Gamma_j^c} c^m  d \gamma_n$ and $\int _{\Gamma_j^c} c^m  d \bar \gamma$ by $m  \epsilon$, an negligible term when we take $\epsilon \to 0$. Consequently, to prove (\ref{eq-lemdp-prf3}), we can focus on the integrals of $c^m$ on the compact set $\Gamma_j$ and on
bounding the difference
\begin{equation} \label{eq-lemdp-prf4}
   \int_{\Gamma_j} c^m  d \gamma_n - \int _{\Gamma_j} c^m  d \gamma.
\end{equation}    
 \item[3.] We now handle the term (\ref{eq-lemdp-prf4})---this is where we apply Lusin's theorem and the majorization property given in Assumption~\ref{cond-pc-3}(M). 
Corresponding to $\Gamma_j$, let us choose the element $(O, D, \nu, B) : = (O_j, D_j, \nu_j, B^1_{j,m,\ell}) \in \W_1$ (cf.\ the definition of the set $\W_1$ given in Step (ii)). By the definition of the set $B^1_{j,m,\ell}$ (cf.~Lemma~\ref{lem-lusin} in Step (ii)), the function $c^m$ is continuous on the closed set $B \times F$, where $F = \proj_\A(\Gamma_j)$, and $\nu(B^c) \leq \delta = \ell^{-1}$. 
We handle the continuous part of $c^m$ separately from the rest of $c^m$. 
Specifically, we first consider the restriction of $c^m$ to the closed set $(D \cup B) \times F$, which is a lower semicontinuous function on $(D \cup B) \times F$ in view of the property of $D$ given in Assumption~\ref{cond-pc-3}(M).
We apply the Tietze--Urysohn extension theorem \cite[Thm.~2.6.4]{Dud02} to extend this function to a function $\tilde c^m$ on the entire space $\X \times \A$ that is nonnegative, lower semicontinuous, and also bounded above by $m$. 
Since $\gamma_n \wto \bar \gamma$, by \cite[Prop.~E.2]{HL96},
$$ \liminf_{n \to \infty} \int \tilde c^m d \gamma_n \geq \int \tilde c^m \, d \bar \gamma.$$
We then handle the difference between $c^m$ and $\tilde c^m$. These two functions differ only outside the set $(D \cup B) \times F$. By using the fact $\nu(B^c) \leq \delta$ and $O \supset \proj_\X(\Gamma_j)$ (cf.\ Assumption~\ref{cond-pc-3}(M)), the majorization property given in Lemma~\ref{lem-su-mp}, and the bounds $\int_{\Gamma_j^c} c^m d \gamma_n \leq m \epsilon$, $\int _{\Gamma_j^c} c^m d \bar \gamma \leq m  \epsilon$ from Step 2, we can calculate that 
$$ \limsup_{n \to \infty} \left| \int_{\X \times \A} ( c^m - \tilde c^m)  d \gamma_n  \right| \leq m \, (\delta +\epsilon), \qquad \ \ \left| \int_{\X \times \A} ( c^m - \tilde c^m)  d \bar \gamma  \right| \leq m \, (\delta +\epsilon).$$
\item[4.] Finally, putting all the pieces together gives us the inequality
$$ \liminf_{n \to \infty} \int c^m \, d \gamma_{n} \geq \int c^m \, d \bar \gamma - 2m \, ( \delta + \epsilon).$$
By letting $\ell \to \infty$ so that $\delta, \epsilon \to 0$, the desired relation (\ref{eq-lemdp-prf3}) follows for all $m \in \Z_+$ and this implies (\ref{eq-lemdp-prf1}).
\end{enumerate} 

\medskip
\noindent {\bf Proving (\ref{eq-lemdp-prf2}):}
The proof of (\ref{eq-lemdp-prf2}) is similar to the above and essentially the same as that given in \cite[Sect.~4, proofs of Lems.~4.4,~4.10]{Yu19-minp}. We outline the main arguments below (see \cite{Yu19-minp} for detailed derivations):
\begin{enumerate}[leftmargin=0.7cm,labelwidth=!]
\item[1.] Consider an arbitrary $h \in \hat C_b(\X)$. Let $\epsilon = \delta = \ell^{-1}$, for some arbitrarily large $\ell \in \Z_+$.
Proceed as in Step 2 of the proof of (\ref{eq-lemdp-prf1}) to choose $j \in \Z_+$ large enough so that for the compact set $\Gamma_j$  in Assumption~\ref{cond-pc-3}(SU), we have
$\gamma_n(\Gamma_j^c) \leq \epsilon$ for all $n$ and $\bar \gamma(\Gamma_j^c) \leq \epsilon$. 
\item[2.] Define a function $\phi(x,a) : = \int_\X h(y) q(dy \, |\, x,a)$ on $\X \times \A$. 
Corresponding to the chosen $j$ and $\ell$, choose the element $(O,D,\nu,B) : = (O_j, D_j, \nu_j, B^2_{j,\ell}) \in \W_2$ and let $F : = \proj_\A(\Gamma_j)$. 
By the definition of the set $B^2_{j,\ell}$ (cf.~Lemma~\ref{lem-lusin} in Step (ii)), $\nu(B^c) \leq \delta = \ell^{-1}$ and on the closed set $B \times F$, $q(dy \,|\, \cdot, \cdot)$ is continuous. Then, since $q(dy \,|\, \cdot, \cdot)$ is also continuous on the closed set $D \times \A$ (cf.\ Assumption~\ref{cond-pc-3}(M)) and $h$ is a bounded continuous function, we have, by \cite[Prop.~7.30]{bs}, that the function $\phi$ is continuous on the closed set $(D \cup B) \times F$. We now treat the continuous part of $\phi$ separately: by the Tietze--Urysohn extension theorem \cite[Thm.~2.6.4]{Dud02}, the restriction of $\phi$ to $(D \cup B) \times F$ can be extended to a bounded continuous function $\tilde \phi$ on the entire space $\X \times \A$, with $\| \tilde \phi \|_\infty \leq \| \phi\|_\infty \leq \| h \|_\infty$. Since $\gamma_n \wto \bar \gamma$, we have
$$ \lim_{n \to \infty} \int \tilde \phi \, d \gamma_n = \int \tilde \phi \, d \bar \gamma.$$
We then handle the difference between $\phi$ and $\tilde \phi$. These two functions differ only outside the set $(D \cup B) \times F$. By using the fact $\nu(B^c) \leq \delta$ and $O \supset \proj_\X(\Gamma_j)$ (cf.\ Assumption~\ref{cond-pc-3}(M)), the majorization property given in Lemma~\ref{lem-su-mp}, and the bounds $\gamma_n(\Gamma_j^c) \leq \epsilon$, $\bar \gamma(\Gamma_j^c) \leq \epsilon$ from Step 1, we can calculate that
$$ \limsup_{n \to \infty} \left| \int_{\X \times \A} ( \phi - \tilde \phi ) \, d \gamma_n  \right| \leq 2 \|h\|_\infty \cdot (\delta +\epsilon), \qquad \ \ \left| \int_{\X \times \A} ( \phi - \tilde \phi ) \, d \bar \gamma  \right| \leq 2 \|h\|_\infty \cdot (\delta +\epsilon).$$
\item[3.] Finally, putting all the pieces together gives us the bound
$$ \limsup_{n \to \infty} \left| \int  \phi  \, d \gamma_n  -  \int \phi \, d \bar \gamma  \right| \leq 4 \|h\|_\infty \cdot (\delta +\epsilon).$$
By letting $\ell \to \infty$ so that $\delta, \epsilon \to 0$, the desired relation (\ref{eq-lemdp-prf2}) follows.
\end{enumerate}
%\smallskip

The lemma now follows from (\ref{eq-lemdp-prf1})-(\ref{eq-lemdp-prf2}), as discussed earlier.
\end{proof}
%\smallskip \smallskip

By Lemma~\ref{lem-sp-subv}, $\big((1,\m0), \, \urho \big) = \big(\L \bar \gamma, \, \langle \bar \gamma, \, c \rangle + \bar r \big)$ for 
$\bar r = \urho - \langle \bar \gamma, \, c \rangle  \geq 0$. Thus $\big((1,\m0), \, \urho \big) \in H$ and consequently, $\urho=\rho^*$. This completes the proof of Theorem~\ref{thm-1}.

\subsubsection{Proof of Prop.~\ref{prp-2}}

The proof is similar to that of \cite[Chap.~12.4B, Thm.~12.4.2(c)]{HL99}. 
Since $\{(\rho_n, h_n)\}$ is a maximizing sequence of (\Pstar), for all $n \geq 0$, $(\rho_n, h_n)$ is feasible for (\Pstar):
\begin{equation} \label{eq-prp2-prf4}
  \rho_n + h_n(x) \leq c(x,a) + \int_\X h_n(y) \, q(dy \mid x, a), \qquad \forall \, (x,a) \in \Gamma.
\end{equation} 
By assumption $\rho_n \uparrow \rho^*$ and for each $(x,a) \in \Gamma$, $\int_\X \sup_n | h_n(y) | \, q(dy \,|\, x, a) < + \infty$. The latter implies 
\begin{equation} 
  \limsup_{n \to \infty} \int_\X h_n (x) \, q(dy \mid x, a) \leq \int_\X \limsup_{n \to \infty} h_n (x) \, q(dy \mid x, a) < + \infty \notag
\end{equation}   
by Fatou's lemma. So, letting $n \to \infty$ and taking limit superior on both sides of (\ref{eq-prp2-prf4}), we obtain
\begin{equation} \label{eq-prp2-prf1}
   \rho^* + h^*(x) \leq c(x,a) + \int_\X h^*(y) \, q(dy \mid x, a) < + \infty, \qquad \forall \, (x,a) \in \Gamma,
\end{equation} 
which is the desired inequality and also shows that $h^*$ is finite everywhere.

Next, we prove the ACOE for $p^*$-a.a.\ states. 
Since $(\mu^*, p^*)$ is a stationary minimum pair and  $\int | h^*|  \, d p^* < \infty$ by assumption, we have
\begin{align*}
 \rho^* & = \int_\X \int_\A c(x,a) \, \mu^*(da \mid x) \, p^*(dx),   \\
  - \infty < \int_\X h^*(x) \, d p^* & = \int_\X \int_\A \int_\X h^*(y) \, q(dy \mid x, a) \, \mu^*(da \mid x) \, p^*(dx) < + \infty,
 \end{align*}
and hence
$$ \int_\X \int_\A \left\{ \rho^* + h^*(x) - c(x,a) - \int_\X h^*(y) \, q(dy \mid x, a) \right\} \, \mu^*(da \mid x) \, p^*(dx) = 0.$$
This together with (\ref{eq-prp2-prf1}) implies that for $p^*\!$-a.a.~$x \in \X$,
$$ \rho^* + h^*(x) - \int_\A \left\{ c(x,a) + \int_\X h^*(y) \, q(dy \mid x, a) \right\} \, \mu^*(da \mid x) = 0,$$
which in turn implies that for $p^*\!$-a.a.~$x \in \X$,
\begin{align}
   \rho^* + h^*(x) & = \int_\A \left\{ c(x,a) + \int_\X h^*(y) \, q(dy \mid x, a) \right\} \, \mu^*(da \mid x) \notag \\
       & \geq \inf_{a \in A(x)} \left\{ c(x,a) + \int_\X h^*(y) \, q(dy \mid x, a) \right\}. \label{eq-prp2-prf3}
\end{align}       
Then, by (\ref{eq-prp2-prf1}), equality must hold in (\ref{eq-prp2-prf3}), and this gives the desired ACOE (\ref{eq-ae-acoe}) and (\ref{eq-ae-acoe2}). 
The proof of Prop.~\ref{prp-2} is now complete.

\subsection{Proofs for Section~\ref{sec-4}} \label{sec-4.3}

\subsubsection{Proof of Theorem~\ref{thm-4.1} (Outline)}

The proof of Theorem~\ref{thm-4.1} is similar to that of Theorem~\ref{thm-2.1} on stationary minimum pairs for an unconstrained MDP. The latter proof is given in our prior work \cite[Sect.~4.1, proofs of Prop.~3.2 and Thm.~3.3]{Yu19-minp}, and its main arguments have already been explained earlier in the proof of Lemma~\ref{lem-sp-subv}. So we will only outline the proof of Theorem~\ref{thm-4.1}, in order to avoid repetition. We will first state some of our prior results for unconstrained MDPs. We will then directly apply them to the present case of constrained MDPs.

In \cite[Sect.~4.1]{Yu19-minp} we considered two kinds of sequences $\{\gamma_n\} \subset \P(\Gamma)$. In the first case, $\{\gamma_n\}$ are the occupancy measures of a policy $\pi$, for an initial distribution $\zeta$ that satisfies $J(\pi, \zeta) < \infty$: 
\begin{equation} \label{eq-thm4.1-prf0}
 \gamma_n(B) : = \tfrac{1}{n} \, \textstyle{\sum_{k=1}^n} \, \Pr^\pi_\zeta \big\{ (x_k, a_k) \in B \big\}, \qquad B \in \B(\Gamma).
\end{equation}
In the second case, $\{\gamma_n\}$ corresponds to a sequence of stationary pairs $(\mu_n, p_n)$ that satisfy \linebreak[4] $\sup_{n} J(\mu_n, p_n) < \infty$:
\begin{equation} \label{eq-thm4.1-prf2}
   \gamma_n (d(x,a)) : = \mu_n(da \,|\, x) \, p_n(dx).
\end{equation}
In both cases, $\sup_n \langle \gamma_n, \, c \rangle < \infty$, which, together with the strict unboundedness condition in Assumption~\ref{cond-pc-3}(SU), implies that (i) $\{\gamma_n\}$ is tight and for the compact sets $\Gamma_j$ in Assumption~\ref{cond-pc-3}(SU), as $j \to \infty$, $\gamma_n (\Gamma_j) \to 0$ uniformly in $n$; and (ii) a weakly convergent subsequence $\{\gamma_{n_k}\}$ can be extracted from any subsequence of $\{\gamma_n\}$: $\gamma_{n_k} \wto \bar \gamma \in \P(\Gamma)$. 
For both cases, the limiting probability measure $\bar \gamma$ is proved to have the following properties, by using (i)-(ii) and the majorization condition in 
Assumption~\ref{cond-pc-3}(M): 
\begin{itemize}[leftmargin=0.7cm,labelwidth=!]
\item[(a)] $\bar \gamma$ corresponds to a stationary pair $(\bar \mu, \bar p) \in \Delta_s$ (i.e., $\bar \gamma (d(x,a)) = \bar \mu(da \,|\, x) \, \bar p(dx)$). 
\item[(b)] The average cost of the pair $(\bar \mu, \bar p)$ satisfies
\begin{equation} \label{eq-thm4.1-prf1}
  J(\bar \mu, \bar p) = \langle \bar \gamma, \, c \rangle \leq \liminf_{k \to \infty} \, \langle \gamma_{n_k}, \, c \rangle.
\end{equation}  
\end{itemize}

We now explain how we can apply these results to prove Theorem~\ref{thm-4.1} for the constrained MDP. 
To prove Theorem~\ref{thm-4.1}(i), we consider $\{\gamma_n\}$ defined by (\ref{eq-thm4.1-prf0}) for a pair $(\pi, \zeta) \in \S$. 
By the feasibility of $(\pi, \zeta)$, its average costs are all finite: 
$$J_i(\pi,\zeta) = \limsup_{n \to \infty} \, \langle \gamma_n, \, c_i \rangle  < \infty, \qquad \forall \, i = 0, 1, \ldots, d.$$ 
Since at least one of the one-stage cost functions $c_0, c_1, \ldots, c_d$ is strictly unbounded by Assumption~\ref{cond-su-cmdp}(SU), 
this implies that $\{\gamma_n\}$ is a tight family of probability measures on $\Gamma$ and for the compact sets $\Gamma_j$ in Assumption~\ref{cond-su-cmdp}(SU), the convergence $\gamma_n(\Gamma_j) \to 0$ as $j \to \infty$ is uniform in $n$.
We then proceed as in the unconstrained case to obtain, from a weakly convergent subsequence $\{\gamma_{n_k}\}$ of $\{\gamma_n\}$, the limiting probability measure $\bar \gamma$. 
Next, using the majorization condition in Assumption~\ref{cond-su-cmdp}(M), it follows as before that $\bar \gamma$ has the property (a) given above and gives us a stationary pair $(\bar \mu, \bar p)$. Moreover, because Assumption~\ref{cond-su-cmdp}(M) is the same as Assumption~\ref{cond-pc-3}(M) holding for every one-stage cost function $c_i$ in the constrained MDP, (\ref{eq-thm4.1-prf1}) in the property (b) above now holds with the function $c$ replaced by every $c_i$; that is
$$   J_i(\bar \mu, \bar p) = \langle \bar \gamma, \, c_i \rangle \leq \liminf_{k \to \infty} \, \langle \gamma_{n_k}, \, c_i \rangle, \qquad \forall \, i = 0, 1, \ldots, d.$$ 
Since 
$J_i(\pi, \zeta) = \limsup_{n \to \infty} \langle \gamma_{n}, \, c_i \rangle \geq \liminf_{k \to \infty} \, \langle \gamma_{n_k}, \, c_i \rangle$,
it follows that
$$ J_i(\bar \mu, \bar p) \leq J_i(\pi, \zeta), \qquad \forall \, i = 0, 1, \ldots, d.$$
This proves Theorem~\ref{thm-4.1}(i). 

To prove Theorem~\ref{thm-4.1}(ii), which asserts the existence of a stationary optimal pair, we consider a sequence of stationary pairs $(\mu_n, p_n) \in \S$ with $J_0(\mu_n, p_n) \downarrow \rho_c^*$ (there exists such a sequence by the part (i) just proved).
Let $\gamma_n$ be defined as in (\ref{eq-thm4.1-prf2}). Then, since $J_i(\mu_n, p_n) = \langle \gamma_n, \, c_i \rangle$ for all $0 \leq i \leq d$, we have
$$  \sup_{n \geq 0} \, \langle \gamma_n, \, c_i \rangle < \infty, \qquad   \forall \, i = 0, 1, \ldots, d.$$
Since at least one of the functions $c_i$ is strictly unbounded under our assumption, as in the proof of the part (i), we can extract a weakly convergent subsequence $\{\gamma_{n_k}\}$ of $\{ \gamma_n\}$ and from its limiting probability measure $\gamma^*$, we can obtain a stationary pair $(\mu^*, p^*)$ such that for all $i = 0, 1, \ldots, d$,
\begin{equation} \label{eq-thm4.1-prf3}
 J_i(\mu^*, p^*) = \langle \gamma^*, \, c_i \rangle \leq \liminf_{k \to \infty} \, \langle \gamma_{n_k}, \, c_i \rangle. 
\end{equation} 
Since $\langle \gamma_{n_k}, \, c_i \rangle =J_i(\mu_{n_k}, p_{n_k})$ and $(\mu_{n_k}, p_{n_k})$ is feasible for the constrained problem, (\ref{eq-thm4.1-prf3}) implies
$$ J_0(\mu^*, p^*) \leq \rho_c^*, \qquad J_i(\mu^*, p^*) \leq \kappa_i, \qquad \forall \,  i = 1, 2, \ldots, d.$$
Hence $(\mu^*, p^*)$ is a stationary optimal pair for the constrained MDP.

We now prove Theorem~\ref{thm-4.1}(iii), which asserts the existence of a stationary lexicographically optimal pair. 
First, let us define recursively sets $\S^*_i$ and scalars $\kappa^*_i$ as follows:
Let 
$$  \kappa^*_0 : = \rho_c^*, \qquad \S^*_0 : = \big\{ (\pi, \zeta) \in \S \mid J_0(\pi, \zeta) = \rho_c^* \big\},  $$
and for $1 \leq i \leq d$, let
$$  \kappa^*_i : = \inf \big\{ J_i(\pi, \zeta) \, \big| \,  (\pi, \zeta) \in \S^*_{i-1} \big\}, \qquad  \S^*_i : = \big\{ (\pi, \zeta) \in \S^*_{i-1} \mid J_i(\pi, \zeta) = \kappa^*_i \big\}.$$
Then $\S \supset \S^*_0 \supset \S^*_1 \cdots \supset \S^*_d$ and $\S^*_d$ consists of all the lexicographically optimal pairs. So, to prove Theorem~\ref{thm-4.1}(iii), we need to show $\Delta_s \cap \S^*_d \not=\varnothing$. 
By Theorem~\ref{thm-4.1}(ii) just proved, $\Delta_s \cap \S^*_0 \not= \varnothing$. 
Let us prove by induction that $\Delta_s \cap \S^*_i \not= \varnothing$ for all $i \leq d$.

Assume that for some $j \leq d$, $\S^*_{j-1} \not= \varnothing$. Then $\kappa^*_j$ is well-defined,
and there exists a sequence of policy and initial distribution pairs $(\pi_n, \zeta_n) \in \S^*_{j-1}$ with
$$ J_j(\pi_n, \zeta_n) \downarrow \kappa^*_j.$$
By Theorem~\ref{thm-4.1}(i) proved earlier, for each $(\pi_n, \zeta_n)$, there is a stationary pair $(\mu_n, p_n)$ with
$$ J_i( \mu_n, p_n) \leq J_i(\pi_n, \zeta_n), \qquad \forall \, i = 0, 1, \ldots, d.$$
This together with the fact $(\pi_n, \zeta_n) \in \S^*_{j-1}$ implies that $(\mu_n, p_n) \in \S^*_{j-1}$. 
Consider now the sequence $\{(\mu_n, p_n)\}$ of stationary pairs thus constructed. Exactly the same proof arguments for establishing the part (ii) can be applied here, and they yield that there exists a stationary pair $(\mu^*, p^*)$ that satisfies (\ref{eq-thm4.1-prf3}). Therefore, 
$$ J_i(\mu^*, p^*)  = \kappa_i^*, \qquad i = 0, 1, \ldots, j,$$
and consequently, $(\mu^*, p^*) \in \S^*_j$. This proves $\Delta_s \cap \S^*_j \not= \varnothing$; then, by induction, $\Delta_s \cap \S^*_d \not= \varnothing$. Hence there is a stationary lexicographically optimal pair for the constrained MDP.

This completes the proof of Theorem~\ref{thm-4.1}.

\subsubsection{Proof of Theorem~\ref{thm-4.2} (Outline)}

The consistency and solvability of (P) follow from Theorem~\ref{thm-4.1}(i)-(ii). The consistency of (\Pstar) is trivial (e.g., let $\rho=0, h(\cdot) \equiv0$, $\beta = 0$). Thus, $0 \leq \sup (\Pstar) \leq \inf (\text{P}) = \rho_c^*$.

We now prove the absence of a duality gap. This proof is similar to that of Theorem~\ref{thm-1}(ii) for the unconstrained MDP case. Since the value of (\Pstar) is finite, by \cite[Thm.~3.3]{AnN87} (cf.\ Theorem~\ref{thm-subc}), the value of (\Pstar) equals the subvalue $\urho$ of (P). 
Therefore, to prove there is no duality gap is to prove $\urho = \rho_c^*$. For this, it suffices to show
$$\big((1, \m0, \kappa), \urho \big) \in H, $$
where the set $H$ is as defined in (\ref{eq-H}) and, for the case here, is given by
\begin{equation}
 H : = \big\{ \big(L (\gamma, \alpha), \, \langle \, \gamma, \, c_0 \, \rangle + r \big) \, \big| \, \gamma \in \Me_w^+(\Gamma), \ \alpha \in \R_+^d, \ r \geq 0 \big\}.
\end{equation}
Recall that by definition the subvalue $\urho = \inf \big\{ r \mid \big( (1, \m0, \kappa), r \big) \in \overbar{H} \big\}$ (cf.\ Section~\ref{sec-2.2}).

To prove $\big((1, \m0, \kappa), \urho \big) \in H$, we will construct a stationary pair $(\bar \mu, \bar p) \in \S$ with $J_0(\bar \mu, \bar p) \leq \urho$, and the proof proceeds in four steps as in the proof of Theorem~\ref{thm-1}(ii). 
Let us outline these steps, explaining briefly some minor changes in the details of the arguments.

\medskip
\noindent {\bf Step (i):} From the definition of $\urho$, it follows that $\big((1, \m0, \kappa), \urho \big) \in \overbar{H}$ and
there exist a direct set $\nI$ and a net $\{(\gamma_i, \alpha_i)\}_{i \in \nI}$ in $\Me_w^+(\Gamma) \times \R_+^d$ such that
\begin{align} 
    \gamma_i(\Gamma) & \to 1,   \label{eq-thm4.2-prf1} \\
  \int_\X h(x) \, \hat \gamma_i(dx) - \int_{\Gamma} \int_\X h(y) \, q(dy \mid x, a) \, \gamma_i(d(x,a)) & \to 0, \qquad \forall \, h \in \Fn_b(\X),  \label{eq-thm4.2-prf2} \\
\big( \langle \gamma_i, \, c_1 \rangle, \, \cdots, \, \langle \gamma_i, \, c_d \rangle \big) + \alpha_i & \to \kappa,  \label{eq-thm4.2-prf2b} \\
 \langle \, \gamma_i, c_0 \, \rangle & \to \urho.  \label{eq-thm4.2-prf3}
\end{align}
As before, in view of (\ref{eq-thm4.2-prf1}) and the fact $\gamma_i \in \Me_w^+(\Gamma)$, by redefining the net $\{(\gamma_i, \alpha_i)\}_{i \in \nI}$ if necessary, we may assume that every $\gamma_i$ in the above is a probability measure on $\B(\Gamma)$.

\medskip
\noindent {\bf Step (ii):} Similarly to Lemma~\ref{lem-net2seq}, we extract a \emph{sequence} $\{(\gamma_n, \alpha_n)\}_{n \geq 0} \subset \{(\gamma_i, \alpha_i)\}_{i \in \nI}$ such that
\begin{align}
  \int_\X h(x) \, \hat \gamma_n(dx) - \int_{\Gamma} \int_\X h(y) \, q(dy \mid x, a) \, \gamma_n(d(x,a)) & \to 0, \qquad \forall \, h \in \hat \C_b(\X) \cup \hat \Fn_b(\X),  \label{eq-thm4.2-prf2a} \\
  \big( \langle \gamma_n, \, c_1 \rangle, \, \cdots, \, \langle \gamma_n, \, c_d \rangle \big) + \alpha_n & \to \kappa,  \label{eq-thm4.2-prf2ab} \\
   \langle \, \gamma_n, c_0 \, \rangle  & \to \urho,  \label{eq-thm4.2-prf3a}
\end{align}
where $\hat \C_b(\X)$ and $\hat \Fn_b(\X)$ in (\ref{eq-thm4.2-prf2a}) are two chosen \emph{countable} subsets of $\Fn_b(\X)$, the properties of which are needed in the subsequent two steps of our proof. In particular, the set $\hat \C_b(\X)$ is the countable set of bounded continuous functions with the property (\ref{eq-detclass}), the same set as defined in the proof of Theorem~\ref{thm-1}(ii).
The countable set $\hat \Fn_b(\X)$ is also defined by the equation (\ref{eq-def-indfn}) in that proof:
$$  \hat \Fn_b(\X) : = \big\{ \ind_E(\cdot) \mid  E = (O \setminus D) \cap B^c \ \, \text{for some} \ \, (O, D, \nu, B) \in \W_1 \cup \W_2 \big\}.$$
However, while the set $\W_2$ is defined in the same way as before, we define the set $\W_1$ slightly differently here, to take into account the multiple one-stage cost functions in the constrained MDP. 
Specifically, in the definition of $\W_1$ (cf.\ Lemma~\ref{lem-lusin} and the definitions preceding this lemma), we make the following changes. We now use the sets and finite measures $(O,D,\nu)$ involved in Assumption~\ref{cond-su-cmdp}(M) instead of Assumption~\ref{cond-pc-3}(M). We choose the sets $B^1_{j,m,\ell}$ for each $j, m, \ell \in \Z_+$ such that besides the property in Lemma~\ref{lem-lusin}(i), we have that restricted to $B^1_{j,m,\ell} \times F_j$, \emph{all} the $(d+1)$ truncated one-stage cost functions, $c_i^m$, $i = 0, 1, \ldots, d$, are continuous (where $c_i^m(\cdot) = \min\{ c_i(\cdot), m\}$). This is possible by Lusin's theorem (since we have only a finite number of these cost functions, we can apply Lusin's theorem to each one of them and then combine the results).

\medskip
\noindent {\bf Step (iii):} This step is the same as before. The relations (\ref{eq-thm4.2-prf2ab})-(\ref{eq-thm4.2-prf3a}) together with Assumption~\ref{cond-su-cmdp}(SU) imply that $\{\gamma_n\}$ is a tight family of probability measures and therefore has a weakly convergent subsequence $\{\gamma_{n_k}\}$. Consider the corresponding subsequence $\{(\gamma_{n_k}, \alpha_{n_k})\}$; for notational simplicity, we will drop the subscript $k$ by redefining $\{(\gamma_n, \alpha_n)\}$ to be this subsequence.
Now, denote the limit of $\{\gamma_n\}$ by $\bar \gamma$, and decompose $\bar \gamma$ as $\bar \gamma(d(x,a)) = \bar \mu(da \,|\, x) \, \bar p(dx)$, where $\bar p$ is the marginal of $\bar \gamma$ on $\X$ and $\bar \mu$ is a stationary policy. Then, using Assumption~\ref{cond-su-cmdp}(M) instead of Assumption~\ref{cond-pc-3}(M), we have that Lemma~\ref{lem-su-mp} holds as before, which gives us the desired majorization properties for $\hat \gamma_n$ and $\bar p$ that we will need in the next, last step.

\medskip
\noindent {\bf Step (iv):} This step is almost the same as before, except that we apply those arguments in the proof of (\ref{eq-lemdp-prf1}) to \emph{every} cost function $c_i$, $0 \leq i \leq d$, in the present constrained problem. 
Then, similar to Lemma~\ref{lem-sp-subv}, we obtain that the pair $(\bar \mu, \bar p)$ is a stationary pair and satisfies that
$$ J_i(\bar \mu, \bar p)  = \langle \bar \gamma, \, c_i \rangle \leq \liminf_{n \to \infty} \, \langle \gamma_{n}, c_i \rangle, \qquad \forall \, i = 0, 1, \ldots, d.$$
Combining this with (\ref{eq-thm4.2-prf2ab}) and (\ref{eq-thm4.2-prf3a}) (recall also $\alpha_n \geq 0$), we obtain
$$ J_0(\bar \mu, \bar p) \leq \urho, \qquad J_i(\bar \mu, \bar p) \leq \kappa_i , \qquad \forall \, i = 0, 1, \ldots, d.$$
Therefore, if we let
$$ \bar r : = \urho - J_0(\bar \mu, \bar p) \geq 0, \qquad \bar \alpha : = \kappa - \big(J_1(\bar \mu, \bar p), \, \ldots, \, J_d(\bar \mu, \bar p) \big) \geq 0,$$
then 
$$ \big((1, \m0, \kappa), \urho \big) = \big( \L( \bar \gamma, \bar \alpha), \, \langle \bar \gamma, \, c_0 \rangle + \bar r \big) \in H.$$
This implies $\urho = \rho^*_c$ (since it implies $\rho^*_c \leq \urho$, whereas $\urho \leq \rho^*_c$). Hence there is no duality gap between (P) and (\Pstar).

\subsubsection{Proofs of Props.~\ref{prp-4.3} and \ref{prp-4.4}}

\smallskip
\begin{proof}[Proof of Prop.~\ref{prp-4.3}]
(i) Consider any $j \in \J^{(0)}$ and some pair $(\pi, \zeta) \in \S$ with $J_j(\pi, \zeta) < \kappa_j$.
By Theorem~\ref{thm-4.1}(i), there exists a stationary pair $(\bar \mu, \bar p) \in \S$ with $J_i(\bar \mu, \bar p) \leq J_i(\pi, \zeta)$ for all $0 \leq i \leq d$. 
Then $J_j(\bar \mu, \bar p) < \kappa_j$. 

Now for each $n \geq 0$, since $(\rho_n, h_n, \beta_n)$ is feasible for (\Pstar), we have from (\ref{eq-cmdp-dual-cstr1})-(\ref{eq-cmdp-dual-cstr2}) that $\beta_n \leq 0$ and for all $(x,a) \in \Gamma$,
$$ \rho_n + h_n(x) \leq c_0(x,a) - \sum_{i=1}^d \beta_{n,i} c_i(x,a) + \int_{\X} h_n(y) \, q(dy \mid x, a),$$
and therefore, by adding $\sum_{i=1}^d \beta_{n,i} \kappa_i$ to both sides,
\begin{equation} \label{eq-prp4.3-prf0}
 \rho_n +  \sum_{i=1}^d \beta_{n,i} \kappa_i + h_n(x) \leq  c_0(x,a) + \sum_{i=1}^d \beta_{n,i} \big( \kappa_i  - c_i(x,a) \big) + \int_{\X} h_n(y) \, q(dy \mid x, a).
\end{equation} 
Integrate both sides of (\ref{eq-prp4.3-prf0}) w.r.t.\ the probability measure $\bar \gamma(d(x,a)) = \bar \mu(da \,|\, x) \, \bar p(dx)$. 
Notice that $\int h_n \, d \bar p = \int_{\Gamma} \int_{\X} h_n(y) \, q(dy \,|\, x, a) \, \bar \gamma(d(x,a))$ since $(\bar \mu, \bar p)$ is a stationary pair. We thus obtain
\begin{equation}
 \rho_n + \sum_{i=1}^d \beta_{n,i} \kappa_i \leq J_0(\bar \mu, \bar p) + \sum_{i=1}^d \beta_{n,i} \big( \kappa_i - J_i(\bar \mu, \bar p) \big). \label{eq-prp4.3-prf0a}
\end{equation}
Take $n \to \infty$. Since $\{(\rho_n, h_n, \beta_n)\}$ is a maximizing sequence for (\Pstar), $\rho_n +  \sum_{i=1}^d \beta_{n,i} \kappa_i \to \rho^*_c$ by Theorem~\ref{thm-4.2}(ii). It then follows from (\ref{eq-prp4.3-prf0a}) that
\begin{align}
    \rho_c^* - J_0(\bar \mu, \bar p) & \leq \liminf_{n \to \infty} \, \sum_{i=1}^d \beta_{n,i} \big( \kappa_i - J_i(\bar \mu, \bar p) \big) \label{eq-prp4.3-prf1} \\
       &  \leq \liminf_{n \to \infty}  \, \beta_{n,j} \big( \kappa_j - J_j(\bar \mu, \bar p) \big), \label{eq-prp4.3-prf2}
\end{align}
where we used the fact $ \kappa_i - J_i(\bar \mu, \bar p) \geq 0$ and $\beta_{n,i} \leq 0$ for all $i$ to derive (\ref{eq-prp4.3-prf2}).
Since $\kappa_j - J_j(\bar \mu, \bar p) > 0$, (\ref{eq-prp4.3-prf2}) implies $\liminf_{n \to \infty} \beta_{n,j} > - \infty$. Hence the sequence $\{\beta_{n,j}\}_{n \geq 0}$ is bounded.

\medskip
\noindent (ii) In this case, suppose $j$ is such that $J_j(\mu^*, p^*) < \kappa_j$. Then $j \in \J^{(0)}$ and (\ref{eq-prp4.3-prf2}) holds with $(\bar \mu, \bar p) = (\mu^*, p^*)$ and with its left-hand side equal to $\rho_c^* - J_0(\mu^*, p^*) = 0$. This yields $\lim_{n \to \infty} \beta_{n,j} = 0$.

\medskip
\noindent (iii) In this case, by assumption there is some pair $(\bar \pi, \bar \zeta) \in \Pi \times \P(\X)$ satisfying 
$$ J_j(\bar \pi, \bar \zeta) < \kappa_j \quad \forall \, j \in \J^{(1)}, \qquad J_j(\bar \pi, \bar \zeta) < \infty \quad \forall \, j \in \J^{(0)} \cup \{ 0\}. $$
As in part (i), let us consider a stationary pair $(\bar \mu, \bar p)$ with $J_i(\bar \mu, \bar p) \leq J_i(\bar \pi, \bar \zeta)$ for all $0 \leq i \leq d$. 
Such a pair exists by Theorem~\ref{thm-4.1}(i), since we can apply this theorem with a different feasible set $\S'$ instead of $\S$ and in $\S'$ we can use $J_i(\bar \pi, \bar \zeta)$ as the upper limits on the average costs w.r.t.\ $c_i$ for $1 \leq i \leq d$, for instance. 

The average costs of this stationary pair $(\bar \mu, \bar p)$ thus satisfy
\begin{equation}
  J_j(\bar \mu, \bar p) < \kappa_j \quad \forall \, j \in \J^{(1)}, \qquad J_j(\bar \mu, \bar p) < \infty \quad \forall \, j \in \J^{(0)} \cup \{ 0\}. \label{eq-prp4.3-prf3}
\end{equation}  
We also have, as in part (i), that (\ref{eq-prp4.3-prf1}) holds for this pair $(\bar \mu, \bar p)$.
Now, as we proved in part (i), $\{\beta_{n,i}\}_{n \geq 0}$ is bounded for every $i \in \J^{(0)}$. This together with the second relation in (\ref{eq-prp4.3-prf3}) implies that the term
$$\limsup_{n \to \infty} \sum_{ i \in \J^{(0)}} \beta_{n,i} \big( \kappa_i - J_i(\bar \mu, \bar p) \big)$$ 
is finite. From (\ref{eq-prp4.3-prf1}), we have the inequality
\begin{align}
  \rho_c^* - J_0(\bar \mu, \bar p) & \leq \liminf_{n \to \infty} \, \sum_{i=1}^d \beta_{n,i} \big( \kappa_i - J_i(\bar \mu, \bar p) \big) \notag \\
        & \leq   \limsup_{n \to \infty} \sum_{ i \in \J^{(0)}} \beta_{n,i} \big( \kappa_i - J_i(\bar \mu, \bar p) \big)  + \, \liminf_{n \to \infty} \sum_{ i \in \J^{(1)}} \beta_{n,i} \big( \kappa_i - J_i(\bar \mu, \bar p) \big). \label{eq-prp4.3-prf4}
\end{align}
In (\ref{eq-prp4.3-prf4}), since the term on the left-hand side and the first term on the right-hand side are both finite, the second term on the right-hand side must satisfy
$$  \liminf_{n \to \infty} \sum_{ i \in \J^{(1)}} \beta_{n,i} \big( \kappa_i - J_i(\bar \mu, \bar p) \big) > - \infty. $$
Then, since $\beta_n \leq 0$, in view of the first relation in (\ref{eq-prp4.3-prf3}), the preceding inequality implies that $\{\beta_{n,i}\}_{n \geq 0}$ must be bounded for every $i \in \J^{(1)}$. Combining this with the result of part (i), we obtain that for every $i = 1, 2, \ldots, d$, the sequence $\{\beta_{n,i}\}_{n \geq 0}$ is bounded. Hence $\{\beta_n\}$ is bounded.
\end{proof}
%\smallskip \smallskip

\begin{proof}[Proof of Prop.~\ref{prp-4.4}]
The proof arguments are similar to those of \cite[Thm.~5.2(b)]{HGL03} for constrained MDPs and those of Prop.~\ref{prp-2} for unconstrained MDPs. 
By the feasibility of $\{(\rho_n, h_n, \beta_n)\}$ for (\Pstar), we have the inequality (\ref{eq-prp4.3-prf0}); that is, for each $n \geq 0$,
$$\quad \rho_n +  \sum_{i=1}^d \beta_{n,i} \kappa_i + h_n(x) \leq  c_0(x,a) + \sum_{i=1}^d \beta_{n,i} \big( \kappa_i  - c_i(x,a) \big) + \int_{\X} h_n(y) \, q(dy \mid x, a), \ \ \  \forall \, (x,a) \in \Gamma.$$
Let $n \to \infty$. Since $\rho_n + \sum_{i=1}^d \beta_{n,i} \kappa_i \uparrow \rho_c^* < \infty$ and $\beta_n \to \beta^* \leq 0$ by assumption, we obtain
$$ \rho_c^* + \limsup_{n \to \infty} h_n(x) \leq c_0(x,a) +  \sum_{i=1}^d \beta^*_{i} \big( \kappa_i  - c_i(x,a) \big) + \limsup_{n \to \infty} \int_{\X} h_n(y) \, q(dy \mid x, a), \ \ \ \forall \, (x,a) \in \Gamma.$$
For each $(x,a) \in \Gamma$, it follows from the assumption $\int_\X \sup_{n \geq 0} | h_n(y) | \, q(dy \,|\, x, a) < + \infty$ and Fatou's lemma that
\begin{equation} 
  \limsup_{n \to \infty} \int_{\X} h_n(y) \, q(dy \mid x, a) \leq \int_{\X} \limsup_{n \to \infty}  h_n(y) \, q(dy \mid x, a) < + \infty. \notag
\end{equation}  
Combining the preceding two relations gives us the desired inequality (\ref{eq-cmdp-dsol}):
\begin{equation} \label{eq-prp4.4-prf2}
  \rho^*_c + h^*(x) \leq c_0(x,a) + \sum_{i=1}^d \beta^*_{i} \big( \kappa_i  - c_i(x,a) \big) +  \int_\X h^*(y) \, q(dy \mid x, a), \qquad \forall \, (x,a) \in \Gamma,
\end{equation}
which also shows that $h^*$ is finite everywhere.

Next, corresponding to the stationary optimal pair $(\mu^*, p^*)$, let $\gamma^*(d(x,a)) = \mu^*(da \,|\, x) \, p^*(dx)$ and
integrate both sides of (\ref{eq-prp4.4-prf2}) w.r.t.\ the probability measure $\gamma^*$. As in the proof of Prop.~\ref{prp-2}, here the integrability is ensured by our assumption $\int |h^*| \, d p^* < \infty$ and the invariance property of $p^*$, which also imply that $- \infty < \int_\X h^*(x) \, d p^*  = \int_\Gamma \int_\X h^*(y) \, q(dy \,|\, x, a) \, \gamma^*\big(d(x,a)\big) < + \infty$. We thus obtain
$$ \rho_c^* \leq J_0(\mu^*, p^*) + \sum_{i=1}^d \beta^*_{i} \big( \kappa_i - J_i(\mu^*, p^*) \big).$$ 
But $J_0(\mu^*, p^*) = \rho_c^*$ and the second term in the right-hand side above is nonpositive, so equality must hold in the above inequality. This result can be equivalently expressed as
$$ \int_\X \int_\A \left\{  \rho^*_c + h^*(x) - c_0(x,a) - \sum_{i=1}^d \beta^*_{i} \big( \kappa_i  - c_i(x,a) \big) -  \int_\X h^*(y) \, q(dy \mid x, a) \right\} \mu^*(da \mid x) \, p^*(dx) = 0.$$
Similarly to the proof of Prop.~\ref{prp-2}, the preceding equality together with the inequality (\ref{eq-prp4.4-prf2}) implies that for $p^*\!$-a.a.\ $x \in \X$,
\begin{align*}
 \rho^*_c - \sum_{i=1}^d \beta^*_{i} \kappa_i +  h^*(x) & = \int_{a \in A(x)} \left\{ c_0(x,a) - \sum_{i=1}^d \beta^*_{i}  c_i(x,a) + \int_\X h^*(y) \, q(dy \mid x, a) \right\} \mu^*(da \mid x) \\
 & = \inf_{a \in A(x)} \left\{ c_0(x,a) - \sum_{i=1}^d \beta^*_{i}  c_i(x,a) + \int_\X h^*(y) \, q(dy \mid x, a) \right\}.
\end{align*} 
This gives the desired ACOE (\ref{eq-cmdp-ae-acoe}) and (\ref{eq-cmdp-ae-acoe2}).
\end{proof}
%\smallskip

\section*{Acknowledgments}
The author would like to thank Professor Eugene Feinberg and the anonymous reviewer for their comments that helped her improve the paper, and Dr.\ Martha Steenstrup for reading parts of the paper and giving her advice on improving the presentation. This research was supported by grants from DeepMind, Alberta Machine Intelligence Institute (AMII), and Alberta Innovates---Technology Futures (AITF).

\addcontentsline{toc}{section}{References} 
\bibliographystyle{apa} 
\let\oldbibliography\thebibliography
\renewcommand{\thebibliography}[1]{%
  \oldbibliography{#1}%
  \setlength{\itemsep}{0pt}%
}
{\fontsize{9}{11} \selectfont
\bibliography{minpair_lp_bib}}

\end{document}